\documentclass[twoside]{amsart}

\numberwithin{equation}{section}
\usepackage{amsmath,amssymb,amsfonts,amsthm,latexsym}
\usepackage{graphicx,color,enumerate}
\usepackage[margin=1in]{geometry}
\usepackage{longtable}

\newtheorem{theorem}{Theorem}[section]
\newtheorem{lemma}[theorem]{Lemma}
\newtheorem{proposition}[theorem]{Proposition}
\newtheorem{corollary}[theorem]{Corollary}

\theoremstyle{definition}
\newtheorem{definition}[theorem]{Definition}
\newtheorem{example}[theorem]{Example}
\theoremstyle{remark}
\newtheorem{remark}[theorem]{Remark}

\newcommand{\C}{\mathbb{C}}
\newcommand{\Sp}{\mathbb{S}}
\newcommand{\Q}{\mathbb{Q}}
\newcommand{\R}{\mathbb{R}}
\newcommand{\T}{\mathbb{T}}

\newcommand{\Z}{\mathbb{Z}}
\newcommand{\N}{\mathbb{N}}
\newcommand{\K}{\mathbb{K}}

\newcommand{\SU}{\operatorname{SU}}
\newcommand{\SL}{\operatorname{SL}}

\newcommand{\Rank}{\operatorname{rank}}

\newcommand{\co}{\colon\thinspace}

\newcommand{\Hilb}{\operatorname{Hilb}}
\newcommand{\on}{\mathit{on}}
\newcommand{\off}{\mathit{off}}

\begin{document}

\title{Hilbert series of symplectic quotients by the $2$-torus}

\author[H.-C.~Herbig]{Hans-Christian Herbig}
\address{
Departamento de Matem\'{a}tica Aplicada,
Universidade Federal do Rio de Janeiro,
Av. Athos da Silveira Ramos 149,
Centro de Tecnologia - Bloco C,
CEP: 21941-909 - Rio de Janeiro, Brazil}
\email{herbighc@gmail.com}

\author[D.~Herden]{Daniel Herden}
\address{
Department of Mathematics,
Baylor University,
Sid Richardson Building,
1410 S.4th Street,
Waco, TX 76706, USA}
\email{daniel\_herden@baylor.edu}

\author[C.~Seaton]{Christopher Seaton}
\address{
Department of Mathematics and Computer Science,
Rhodes College,
2000 N. Parkway,
Memphis, TN 38112, USA}
\email{seatonc@rhodes.edu}

\thanks{
This work was supported by a Collaborate@ICERM grant from the Institute for Computational and Experimental Research in Mathematics (ICERM).
C.S. was supported by the E.C.~Ellett Professorship in Mathematics.
H.-C.H. was supported by CNPq through the \emph{Plataforma Integrada Carlos Chagas}.}


\begin{abstract}
We compute the Hilbert series of the graded algebra of real regular functions
on a linear symplectic quotient by the $2$-torus as well as the first four coefficients of the
Laurent expansion of this Hilbert series at $t = 1$. We describe an algorithm to compute the
Hilbert series as well as the Laurent coefficients in explicit examples.
\end{abstract}

\maketitle
\tableofcontents


\section{Introduction}
\label{sec:Intro}

Let $V$ be a finite-dimensional unitary representation of a compact Lie group $G$. The action of $G$
on the underlying real symplectic manifold of $V$ is Hamiltonian and admits a homogeneous quadratic
moment map. The \emph{symplectic quotient} $M_0$ at the zero level of this moment map is usually singular
but has the structure of a \emph{symplectic stratified space}, i.e., is stratified into smooth symplectic
manifolds; see \cite{SjamaarLerman}. The Poisson algebra of smooth functions on $M_0$ has an $\N$-graded
Poisson subalgebra $\R[M_0]$ of \emph{real regular functions} on $M_0$, the polynomial functions on $M_0$
as a semialgebraic set.

This paper continues a program to compute the Hilbert series of $\R[M_0]$ for various choices of $G$ with
particular attention to the first few coefficients of the Laurent expansion of the Hilbert series around
$1$, here denoted $\gamma_0, \gamma_1, \ldots$. The case when $G = \Sp^1$ is the circle was handled in
\cite{HerbigSeaton}, the case $G = \SU_2$ was treated in \cite{HerbigHerdenSeatonSU2}, and analogous
computations for the Hilbert series of the algebras of off-shell (i.e. classical) invariants were given
in \cite{CowieHerbigSeatonHerden,CayresPintoHerbigHerdenSeaton}.
Here, we consider the case $G = \T^2$, the first step towards understanding those cases where
$\operatorname{rank} G > 1$.

The Hilbert series and its first two Laurent coefficients have played an important role in the study of
classical invariants. Hilbert first computed $\gamma_0$ for irreducible representations of $\SL_2$ in
\cite{Hilbert}, and computations of the Hilbert series or its Laurent coefficients in this case have been
considered by several authors; see for example
\cite{SylvesterFranklin,SpringerInvThryBook,Springer,Brion,LittelmannProcesi,BedratyukPoincareCovariants,
BedratyukBivarPoincare,BedratyukSL2Invariants,BedratyukIlashCovariants,IlashPoincareNLinForm,IlashPoincareNQuadForm}.
When $G$ is finite, it is well known that the first two Laurent coefficients are determined by the
order of $G$ and the number of pseudoreflections it contains; see \cite[Lemma~2.4.4]{SturmfelsBook}.
The meanings of the $\gamma_m$ more generally have been investigated in \cite{AvramovBuchweitzSally} and
\cite[Chapter~3]{PopovBook}.

For symplectic quotients, the Hilbert series continues to be a valuable tool for understanding the
graded algebra of regular functions. Certain properties of a graded algebra, such as Cohen-Macaulayness and
Gorensteinness, can be verified using the Hilbert series \cite{StanleyHilbFunGradAlg}, and this has been
used to check the Gorenstein property for symplectic quotients in \cite{HerbigHerdenSeaton,HerbigSchwarzSeaton2}.
Additionally, the Hilbert series has been used to distinguish between symplectic quotients that are not (graded
regularly) symplectomorphic \cite{FarHerSea,HerbigSeaton2}, and as a heuristic to identify potentially
symplectomorphic symplectic quotients \cite{HerbigLawlerSeaton}.

After reviewing the framework and relevant background information in Section~\ref{sec:Background}, we turn
to the computation of the Hilbert series in Section~\ref{sec:HilbSer}.
The first main result of this paper is Corollary~\ref{cor:HilbDegenerate}, giving a formula for the Hilbert
series of $\R[M_0]$ corresponding to an arbitrary $\T^2$-representation in terms of the weight matrix $A$.
This result is stated in terms of the Hilbert series $\Hilb_A^{\on}(t)$ of an algebra that does not always
coincide with the regular functions $\R[M_0]$ on the symplectic quotient $M_0$ of the representation with weight
matrix $A$ and assumes that $A$ is in a specific \emph{standard form}. However, there is no loss of generality;
we explain in Section~\ref{subsec:T2Reps0} that $\R[M_0]$ can always be computed as $\Hilb_A^{\on}(t)$ for some $A$,
and in Section~\ref{subsec:T2Reps} that $A$ can always be put in standard form with no change to $\Hilb_A^{\on}(t)$.
This approach greatly simplifies the computations in Section~\ref{subsec:HilbGeneric}.
The formula for $\Hilb_A^{\on}(t)$ takes its simplest form in Theorem~\ref{thrm:HilbGeneric} with additional
hypotheses on the representation that are described in Section~\ref{subsec:T2Reps}. The formula suggests a
(not particularly fast) algorithm that we describe in Section~\ref{subsec:Algorithm}.

In Section~\ref{sec:Gammas}, we turn to the computation of the first four Laurent coefficients, which are
given in Theorems~\ref{thrm:Gam0} and \ref{thrm:Gam2}. These computations require results of Smith
\cite{SmithLinCongr} on the number of solutions of a system of linear congruences, which we recall in
Section~\ref{subsec:LinCongruence}. As in the case of $G = \Sp^1$, the resulting formulas have
singularities when certain triples of vectors associated to the columns of the weight matrix are
collinear, in which case we call the weight matrix \emph{degenerate}. We provide a general explanation
for the removability of those singularities in Section~\ref{sec:Gammas} and detail explicit computations
to indicate the nature of the cancellations for the lowest-degree coefficient in Section~\ref{subsec:Cancellations}.
We expect that the numerators of the resulting rational functions admit combinatorial descriptions in terms of some
sort of generalization of Schur polynomials, and such a description would yield closed form expressions for the
Laurent coefficients in the degenerate case. We hope that this paper leads to progress in this direction. Finally,
in Section~\ref{subsec:GamAlg}, we briefly describe methods we have used to efficiently compute the first Laurent
coefficient in the presence of these singularities.


\section*{Acknowledgements}

We express appreciation to the Institute for Computational and Experimental Research in Mathematics (ICERM), Herbig and Seaton express appreciation to Baylor University, and Herden and Seaton express appreciation to the Instituto de Matem\'{a}tica Pura e Aplicada (IMPA) for hospitality during the work contained in this manuscript. Herbig thanks CNPq for financial support. We would also like to thank Anne-Katrin Gallagher for helpful discussion and responses to questions.


\section{Background}
\label{sec:Background}


\subsection{Symplectic quotients associated to representations of $\T^2$}
\label{subsec:T2Reps0}

In this section, we give a concise summary of the construction and relevant background for symplectic quotients
by the $2$-torus. The reader is referred to \cite{HerbigSeaton} for more details; see also
\cite{FarHerSea,HerbigIyengarPflaum,HerbigSeaton2}.

Throughout this paper, we fix the compact Lie group $\T^2$ and consider finite-dimensional unitary
representations $V \simeq \C^n$ of $\T^2$. Such a representation can be described by a \emph{weight matrix}
\[
    A   =   \begin{pmatrix}
                a_{11}  &   a_{12}  &   \cdots  &   a_{1n}  \\
                a_{21}  &   a_{22}  &   \cdots  &   a_{2n}
            \end{pmatrix}
            \in\Z^{2\times n},
\]
where the action of $(z_1,z_2)\in\T^2$ on $(x_1,x_2,\ldots,x_n)\in\C^n$ is given by
\[
    (z_1,z_2)\cdot(x_1,x_2,\ldots,x_n)
    =
    \big( z_1^{a_{11}}z_2^{a_{21}} x_1, z_1^{a_{12}}z_2^{a_{22}} x_2,\ldots,
            z_1^{a_{1n}}z_2^{a_{2n}} x_n \big).
\]
We will often use $V_A$ to indicate that $V$ is the representation with weight matrix $A$, or
simply $V$ when $A$ is clear from the context.
The representation is faithful if and only if $A$ has rank $2$ and the $\gcd$ of the $2\times 2$ minors of
$A$ is equal to $1$; \cite[Lemma~1]{FarHerSea}. Applying to $A$ elementary row operations that are invertible over
$\Z$ corresponds to changing the basis of $\T^2$ and hence does not change the representation.
Note that the $\T^2$-action on $V$ extends to a $(\C^\times)^2$-action with the same description.

With respect to the underlying real manifold of $V$ and symplectic structure compatible with the complex structure,
the action of $\T^2$ is Hamiltonian, and identifying the Lie algebra
$\mathfrak{g}$ of $\T^2$ (and hence its dual) with $\R^2$, the moment map $J^A\co V\to\mathfrak{g}^\ast$ (denoted
$J$ when $A$ is clear from the context) is given by
\[
    J_i^A(x_1,\ldots,x_n) = \frac{1}{2}\sum\limits_{j=1}^n a_{ij} x_j \overline{x_j},
    \quad i=1,2.
\]
The real $\T^2$-invariant variety $Z = Z_A = J^{-1}(\mathbf{0}) \subset V$ is called the \emph{shell}, and the
\emph{symplectic quotient} is the space $M_0 = M_0^A = Z/\T^2$. The symplectic quotient has
a smooth structure given by the Poisson algebra
$\mathcal{C}^\infty(M_0) = \mathcal{C}^\infty(V)^{\T^2}/\mathcal{I}_Z^{\T^2}$,
where $\mathcal{C}^\infty(V)^{\T^2}$ denotes the $\T^2$-invariant smooth $\R$-valued functions on $V$, $\mathcal{I}_Z$
is the ideal of $\mathcal{C}^\infty(V)$ of functions vanishing on $Z$, and
$\mathcal{I}_Z^{\T^2} = \mathcal{I}_Z\cap\mathcal{C}^\infty(V)^{\T^2}$. Equipped with this structure, the symplectic
quotient $M_0$ has the structure of a \emph{symplectic stratified space}, see \cite{SjamaarLerman}.

The algebra $\mathcal{C}^\infty(M_0)$ contains an $\N$-graded Poisson subalgebra $\R[M_0]$ of
real regular functions on $M_0$, whose construction we now describe.
Let $\R[V]^{\T^2}$ denote the graded algebra of $\T^2$-invariant polynomials over $\R$ on $V$.
For emphasis, we will refer to $\R[V]^{\T^2}$ as the \emph{algebra of off-shell invariants}.
After tensoring with $\C$, $\R[V]^{\T^2} \otimes_\R \C$ is isomorphic to
$\C[V\oplus V^\ast]^{\T^2} = \C[V\oplus V^\ast]^{(\C^\times)^2}$, where $V^\ast$ denotes the dual representation;
letting $(y_1,\ldots,y_n)$ denote coordinates for $V^\ast$ dual to the coordinates $(x_1,\ldots,x_n)$, $V$
is the subset of $V\oplus V^\ast$ given by $y_i = \overline{x_i}$ for each $i$.
The weight matrix of the representation $V\oplus V^\ast$ is given by $(A|-A)$, corresponding to the
\emph{cotangent lift} of the original representation. The algebra $\C[V\oplus V^\ast]^{\T^2}$ is generated by
a finite set of monomials which can be computed by the algorithm described in \cite[Section~1.4]{SturmfelsBook}.

We are interested in the quotient $\R[V]^{\T^2}/I_J^{\T^2}$, where $I_J$
is the ideal generated by the components $J_1, J_2$ of the moment map and $I_J^{\T^2} = I_J\cap\R[V]^{\T^2}$
is the invariant part; note that the monomials $x_j \overline{x_j}$ are invariant so that
$J_1, J_2 \in \R[V]^{\T^2}$. The closely related \emph{algebra of real regular functions on $M_0$} is given by
$\R[M_0] = \R[V]^{\T^2}/I_Z^{\T^2}$, where $I_Z$ is the subalgebra of polynomials on $V$ that vanish on $Z$ and
$I_Z^{\T^2} = I_Z\cap\R[V]^{\T^2}$.

For ``sufficiently large" representations $V$, the ideal $I_Z$ is generated by the two components
$J_1, J_2$ of the moment map, i.e., $I_J = I_Z$, which implies that $I_J^{\T^2} = I_Z^{\T^2}$.
This is the case, for example, when the $(\C^\times)^2$-action on $V$ is \emph{stable}, meaning that the principal
isotropy type consists of closed orbits; see \cite[Theorem~3.2 and Corollary~4.3]{HerbigSchwarz}.
When the representation is not stable, there is a stable $(\C^\times)^2$-subrepresentation $V^\prime$ of
$V$ that has the same shell, symplectic quotient, and algebra of real regular functions; see
\cite[Lemma~3]{HerbigSeaton2}; see also \cite[page~10]{FarHerSea} and \cite[Lemma~2]{WehlauPopov}.
As a brief summary of the results in these references applied to the situation at hand: $I_Z$ is generated
by $J_1$ and $J_2$ iff there are no coordinates $x_i$ that vanish identically on the shell, equivalently,
when $A$ can be put in the form $(D|C)$ where $D$ is a $2\times 2$ diagonal matrix with negative diagonal
entries and the entries of $C$ are nonnegative. When this condition fails, $V^\prime$
is constructed by setting to zero any $x_i$ that vanishes on the shell and hence deleting the corresponding
column in $A$. Note in particular that $\R[M_0]$ can always be computed as
$\R[V^\prime]^{\T^2}/I_{J|_{V^\prime}}^{\T^2}$ for a subrepresentation $V'$ of $V$.

The \emph{Hilbert series} of a finitely-generated graded algebra $R = \bigoplus_{d=0}^\infty R_d$ over a field
$\K$ is the generating function of the dimension of $R_d$,
\[
    \Hilb_{R}(t) =   \sum\limits_{d=0}^\infty t^d \dim_{\K} R_d.
\]
The Hilbert series has a radius of convergence of at least $1$ and is the power series of a rational function
in~$t$; see \cite[Section~1.4]{DerksenKemperBook}. For a representation of $\T^2$ as above, we let
$\Hilb_A^{\off}(t)$ denote the Hilbert series of the algebra $\R[V]^{\T^2}$ of off-shell invariants and let
$\Hilb_A^{\on}(t)$ denote the Hilbert series of the algebra $\R[V]^{\T^2}/I_J^{\T^2}$.
By \cite[Lemma~2.1]{HerbigSeaton}, we have the simple relationship
\begin{equation}
\label{eq:OnOffHilb}
    \Hilb_A^{\on}(t)    =   (1 - t^2)^2\Hilb_A^{\off}(t).
\end{equation}
As a consequence, it follows that $\Hilb_A^{\on}(t)$ depends only on the cotangent-lifted weight matrix
$(A|-A)$ and not on $A$. However, it is possible that two representations have isomorphic cotangent-lifted
representations while $I_Z^{\T^2} = I_J^{\T^2}$ for one and not the other. Hence, the algebra of real regular
functions $\R[M_0]$ depends on the representation and not merely the cotangent lift.

\begin{example}
\label{ex:LagrangSection}
Let
\[
    A   =   \begin{pmatrix} -1 & 0 & -1 \\ 0 & -1 & -1 \end{pmatrix}
\]
and
\[
    B   =   \begin{pmatrix} -1 & 0 & 1 \\ 0 & -1 & 1 \end{pmatrix}.
\]
Then the cotangent lift of the representation with weight matrix $A$ has weight matrix
\[
    (A|-A)  =   \left(\begin{array}{ccc|ccc}
                -1 & 0  & -1 & 1 & 0 & 1     \\
                0  & -1 & -1 & 0 & 1 & 1 \end{array}\right),
\]
which is clearly isomorphic to the cotangent lift with weight matrix $(B|-B)$ by simply permuting columns.
The moment map associated to $A$ is
\begin{align*}
    J_1^A(x_1,x_2,x_3)  &=  - \frac{1}{2}\big( x_1 \overline{x_1} + x_3\overline{x_3} \big)
    \\
    J_2^A(x_1,x_2,x_3)  &=  - \frac{1}{2}\big( x_2 \overline{x_2} + x_3\overline{x_3} \big),
\end{align*}
so that the corresponding shell $Z_A$ is the origin and the symplectic quotient $M_0^A$ is a point.
Because each $x_i$ vanishes on the shell, the representation $V^\prime$ is the origin, and $\R[M_0^A]$
is given by $\R[V_A^\prime]^{\T^2}/I_{J^A|_{V^\prime}}^{\T^2} = \R$.

However, the moment map associated to $B$ is
\begin{align*}
    J_1^B(x_1,x_2,x_3)  &=  \frac{1}{2}\big( - x_1 \overline{x_1} + x_3\overline{x_3} \big)
    \\
    J_2^B(x_1,x_2,x_3)  &=  \frac{1}{2}\big( - x_2 \overline{x_2} + x_3\overline{x_3} \big),
\end{align*}
and the shell $Z_B$ has real dimension $4$ and $M_0^B$ has real dimension $2$. In this case,
each $x_i$ obtains a nonzero value on the shell, and $\R[M_0^B]$ is equal to the algebra
$\R[V_B]^{\T^2}/I_{J^B}^{\T^2}$.
\end{example}

Representations with weight matrices $A$ and $B$ are equivalent if $B$ can be obtained from $A$
by permuting columns and elementary row operations over $\Z$. For the cotangent-lift, because transposing a column
of $A$ with the corresponding column of $-A$ corresponds to multiplying the column by $-1$,
the representations corresponding to $(A|-A)$ and $(B|-B)$ are equivalent if $B$ can be obtained
from $A$ by permuting columns, elementary row operations over $\Z$, and multiplying columns by $-1$. In the sequel,
we will take advantage of this fact and put $A$ into a standard form given in
Definition~\ref{def:Generic}. Note that, if we begin with a weight matrix $B$ such that
$\R[M_0^B] = \R[V_B]^{\T^2}/I_{J^B}^{\T^2}$, replacing $B$ with a matrix $A$ in standard form may break this
relationship; we may have $\R[M_0^A] \neq \R[V_A]^{\T^2}/I_{J^A}^{\T^2}$ as in Example~\ref{ex:LagrangSection}
above. However, as $\Hilb_B^{\on}(t)$ depends only on the cotangent lift, we still have
$\Hilb_B^{\on}(t) = \Hilb_A^{\on}(t)$. That is, the change to standard form may cause $\Hilb_A^{\on}(t)$
to no longer describe $\R[M_0^A]$, but it still describes the algebra $\R[M_0^B]$ associated to the symplectic
quotient associated to $B$. For this reason, we state our results in terms of $\Hilb_A^{\on}(t)$ in standard
form with no loss of generality.


\subsection{Standard form and degeneracies}
\label{subsec:T2Reps}

Let $A \in\Z^{2\times n}$ be the weight matrix of a linear representation of $\T^2$ on $\C^n$.
To avoid trivialities, we assume that there are no trivial subrepresentations, i.e., $A$ has no zero columns.
Let $d_{ij}$ denote the $2\times 2$ minor associated to columns $i$ and $j$, i.e.,
$d_{ij} = a_{1i} a_{2j} - a_{2i} a_{1j}$.
Recall that the $d_{ij}$ satisfy the \emph{Pl\"{u}cker relations} \cite[page 138]{PopovVinberg}. That is,
for any indices $i_0, i_1, i_2$ and $j$, we have
\begin{equation}
\label{eq:Pluecker}
    d_{i_1 i_2} d_{i_0 j} - d_{i_0 i_2} d_{i_1 j} + d_{i_0 i_1} d_{i_2 j} = 0.
\end{equation}

\begin{definition}
\label{def:Generic}
We say that a weight matrix $A\in\Z^{2\times n}$ is:
\begin{enumerate}
\item[(i)]      \emph{faithful} if $\Rank A = 2$ and the $\gcd$ of the set of $2\times 2$ minors of $A$ is $1$;
\item[(ii)]     \emph{in standard form} if $a_{1i} > 0$ for each $i$;
\item[(iii)]    \emph{generic} if it is in standard form, $a_{1i}\neq a_{1j}$ for $i\neq j$, and
                $d_{ij} + d_{ik} + d_{jk} \neq 0$ for each distinct $i,j,k$;
\item[(iv)]     \emph{completely generic} if it is in standard form, generic, and $d_{ij} + d_{jk} + d_{ki} \neq 0$
                for each distinct $i,j,k$; and
\item[(v)]      \emph{degenerate} if it is in standard form and is not generic.
\end{enumerate}
If $A$ is generic, by transposing $i$ and $j$ in the condition $d_{ij} + d_{ik} + d_{jk} \neq 0$,
we also have that $d_{ij} - d_{ik} - d_{jk} \neq 0$ for each distinct $i,j,k$.
\end{definition}

The condition that the weight matrix is faithful is equivalent to the representation being faithful;
see Section~\ref{subsec:T2Reps0}. The condition $d_{ij} + d_{ik} + d_{jk} = 0$ can be interpreted geometrically
as corresponding to the three vectors $\mathbf{a}_i, - \mathbf{a}_j, \mathbf{a}_k \in\R^2$ being collinear,
while $d_{ij} + d_{jk} + d_{ki} = 0$ corresponds to the three vectors $\mathbf{a}_i, \mathbf{a}_j, \mathbf{a}_k$
being collinear. Hence, if $A$ is generic, then for any distinct $i,j,k$, the vectors
$\mathbf{a}_i, - \mathbf{a}_j, \mathbf{a}_k$ are not collinear; if $A$ is completely generic, then for any choice
of $i,j,k$ and any choice of signs, $\pm\mathbf{a}_i, \pm\mathbf{a}_j, \pm\mathbf{a}_k$ are not collinear.

We may assume that $A$ is faithful and in standard form
with no loss of generality, i.e., without changing $\Hilb_A^{\on}(t)$. Specifically, if $A$ is not
faithful, we may replace $\T^2$ with $\T^2/K$ where $K$ is the subgroup acting trivially, yielding a
representation of $\T^{\Rank A}$ with the same symplectic quotient, see \cite[Lemma~2]{FarHerSea}. Similarly,
we may ensure that each $a_{1i} \neq 0$ by adding any but finitely many scalar multiples of the second row to
the first and then can put $A$ in standard form by multiplying columns by $-1$. Note that if $n \leq 2$,
then either the representation is not faithful or there are no nontrivial invariants, so we can assume that
$n > 2$.

It is clear that multiplying columns by $-1$ will change whether $A$ is in standard form, and elementary row operations over
$\Z$ can change a degenerate weight matrix to a generic one. Hence, for representations corresponding to weight
matrices $A$ and $B$ such that the cotangent lifts $(A|-A)$ and $(B|-B)$ describe equivalent representations,
and hence $\Hilb_A^{\on}(t)$ and $\Hilb_B^{\on}(t)$ coincide, it is possible that $A$ is degenerate while
$B$ is generic.

\begin{example}
\label{ex:GenDegenMatrix}
The weight matrix
\[
    A   =   \begin{pmatrix} 2 & 1 & 4 \\ 1 & -1 & 1 \end{pmatrix}
\]
is degenerate as $d_{12} + d_{13} + d_{23} = 0$. Adding twice the second row to the first
and then multiplying the second column by $-1$ yields
\[
    B   =   \begin{pmatrix} 4 & 1 & 6 \\ 1 & 1 & 1 \end{pmatrix},
\]
which is generic. As $(A|-A)$ and $(B|-B)$ are weight matrices of equivalent representations,
$\Hilb_A^{\on}(t) = \Hilb_B^{\on}(t)$.
\end{example}

However, there are degenerate weight matrices that cannot be made generic by these changes of bases, e.g.,
\[
    A   =   \begin{pmatrix} 1 & 1 & 1 \\ 0 & 1 & 1 \end{pmatrix}.
\]

Finally, observe that the condition $d_{ij} + d_{ik} + d_{jk} \neq 0$ for all distinct $i,j,k$ is not invariant
under multiplying columns by $-1$. However, as a consequence of the geometric characterization described above,
the condition that both $d_{ij} + d_{ik} + d_{jk} \neq 0$ and $d_{ij} + d_{jk} + d_{ki} \neq 0$ for all distinct
$i,j,k$ is invariant under multiplying columns by $-1$. Of course, any of these conditions is invariant
under elementary row operations applied to $A$.


\subsection{Counting solutions of systems of linear congruences}
\label{subsec:LinCongruence}

In this section, we recall results concerning the number of solutions of a system of linear congruences due to
Smith \cite{SmithLinCongr}; see \cite{NewmanSmithNormalForm} for a modern discussion.
We begin with the following folklore result; see \cite[Art.~14*, p. 314]{SmithLinCongr} and
\cite[p.~369]{NewmanSmithNormalForm}.

\begin{theorem}
\label{thrm:SmithNormalForm}
Let $A$ be a nonzero $m\times n$ matrix over a PID $R$. Then $A$ can be decomposed into $A = PSQ$, where
$P$ is an invertible $m\times m$ matrix, $Q$ in an invertible $n\times n$ matrix, $S$ is an $m\times n$ matrix
with nonzero entries only on the main diagonal, and the main diagonal entries $a_i$ $(1\le i\le \min\{m,n\})$ of
$S$ satisfy $a_i | a_{i+1}$ for all $i$. In particular, there exists an $1\le r\le \min\{m,n\}$ such that the
values $a_i\neq 0$ for $i\leq r$ and $a_i = 0$ for $i > r$.

The elements $a_i$ are unique up to multiplication by a unit, and the matrix $S$ is called a \emph{Smith normal form} of $A$.
For $R= \Z$, we will assume the canonical choice $a_i \ge 0$.
\end{theorem}

Less well-known is the following additional statement; see \cite[Art.~14*, p.~314]{SmithLinCongr} and
\cite[p.~370]{NewmanSmithNormalForm}.

\begin{proposition}
\label{prop:SmithNormalFormGCD}
Let $A$ be a nonzero $m\times n$ matrix over a PID $R$ with Smith normal form $S$.
For $1\le i\le \min\{m,n\}$, let $\Delta_i$ denote a $\gcd$ of the $i\times i$ minors of $A$.
Then $\Delta_i = \prod_{k=1}^i a_k$ up to multiplication by a unit. In particular, setting $\Delta_0 = 1$, up to multiplication by a unit we have
$a_i = \Delta_{i}/\Delta_{i-1}$ for $1\le i\le r+1$.
\end{proposition}
\begin{proof}
It is easily verified that the row and column operations used to compute the Smith normal form
$S$ from $A$ do not affect the $\Delta_i$. Thus, $A$ and $S$ share the same $\Delta_i$, while
$\Delta_i \sim \prod_{k=1}^i a_k$ is obvious for $S$.
\end{proof}

We then have the following.

\begin{theorem}
[{\cite[Art.~17, p.~320 and Art.~18*, p.~324]{SmithLinCongr}}]
\label{thrm:SmithLinCong}
Let $A$ be a nonzero $m\times n$ matrix over $\Z$ with Smith normal form $S$ and let $N > 1$ be an integer.
Then the number of distinct solutions $\mathbf{x}\in (\Z/N\Z)^n$ of the homogeneous system of congruences
$A \mathbf{x} \equiv \mathbf{0} \mod N$ is
\[
    N^{n-r}\prod_{i=1}^r \gcd(\Delta_{i}/\Delta_{i-1},N)
    =
    N^{n-\min\{m,n\}}\prod_{i=1}^{\min\{m,n\}} \gcd(a_i,N).
\]
\end{theorem}
\begin{proof}
Let $A = PSQ$ denote a Smith decomposition of $A$ over $\Z$. We interpret all matrices over $\Z/N\Z$.
Then $A \mathbf{x} \equiv \mathbf{0}$ is equivalent to $P^{-1}A \mathbf{x}\equiv SQ \mathbf{x} \equiv \mathbf{0}$,
while $\mathbf{x}\mapsto Q\mathbf{x}$ defines an automorphism of $(\Z/N\Z)^n$. In particular, using the
substitution $\mathbf{y}= Q\mathbf{x}$, $A \mathbf{x} \equiv \mathbf{0}$ has as many distinct solutions
as the system of equations $S \mathbf{y} \equiv \mathbf{0}$. For $1\le i \le r$, the equation $a_iy_i \equiv 0 \mod N$ has $\gcd(a_i,N)$ distinct solutions with $a_i = \Delta_{i}/\Delta_{i-1}$. For $r< i \le n$, $y_i$ is a free variable with $N$ distinct solutions.
\end{proof}

With this, we have the following, which will be needed in the sequel.

\begin{proposition}
\label{prop:NumSolsCongruence}
Let $n > 2$, let $A\in\Z^{2\times n}$ be a weight matrix of rank $2$, and let $g$ denote the $\gcd$
of the set of $2\times 2$ minors $d_{ij}$ of $A$. For each $i\neq j$ such that $d_{ij} \neq 0$, the number of pairs $(\xi,\zeta)$ of
$d_{ij}$th roots of unity such that $\xi^{d_{ik}} \zeta^{d_{jk}} = 1$ for each $k\neq i,j$ is given by
$g |d_{ij}|$. In particular, if $A$ is faithful, then this number is equal to $|d_{ij}|$.
\end{proposition}
\begin{proof}
First assume that $A$ is faithful so that $\Delta_2 = 1$. By fixing a primitive $d_{ij}$th root of unity
$\xi_0$, we can identify the set of $(\xi,\zeta)$ with $(\Z/d_{ij}\Z)^2$ via
$(\xi,\zeta) = (\xi_0^x,\xi_0^y)$. Then the conditions $\xi^{d_{ik}} \zeta^{d_{jk}} = 1$ for each $k\neq i,j$
coincide with the system of congruences
\begin{equation}
\label{eq:SystemCongruence}
    \begin{pmatrix}
        d_{i1}  &   d_{j1}
            \\
        d_{i2}  &   d_{j2}
            \\
        \vdots  &   \vdots
            \\
        d_{in}  &   d_{jn}
    \end{pmatrix}
    \begin{pmatrix} x \\ y \end{pmatrix}
    \equiv
    \mathbf{0}
    \mod d_{ij},
\end{equation}
where the rows $(d_{ii}, d_{ji})$ and $(d_{ij}, d_{jj})$ are removed so that the coefficient matrix
is of size $(n-2)\times 2$.

If $n = 3$, then there is only one $k\neq i,j$, so Equation~\eqref{eq:SystemCongruence} is the single
congruence $x d_{ik}  + y d_{jk} \equiv 0 \mod d_{ij}$. By Theorem~\ref{thrm:SmithLinCong},
the number of solutions to this congruence is given by
$|d_{ij}|\gcd(\Delta_1,d_{ij}) = |d_{ij}|\gcd(d_{ik},d_{jk},d_{ij}) = |d_{ij}|$, as $A$ is faithful.

For $n \ge 4$, Theorem~\ref{thrm:SmithLinCong} implies that the number of solutions to
Equation~\eqref{eq:SystemCongruence} is $\gcd(a_2, d_{ij}) \gcd(\Delta_1, d_{ij})$,
where $\Delta_1 = \gcd\{d_{ik}, d_{jk} : k \neq i,j \}$. As $A$ has rank $2$ and $d_{ij}\neq 0$, we have
$\Delta_1\neq 0$, so that $a_2 = \Delta_2/\Delta_1$ where $\Delta_2$ is the $\gcd$ of the $2\times 2$ minors
of the $(n-2)\times 2$ coefficient matrix of Equation~\eqref{eq:SystemCongruence}; hence,
the number of solutions is equal to $\gcd(\Delta_2/\Delta_1, d_{ij}) \gcd(\Delta_1, d_{ij})$.

Applying the Pl\"{u}cker relations, Equation~\eqref{eq:Pluecker} with $i_0 = k_2$, $i_1 = i$, $i_2 = k_1$,
we have that the $2\times 2$ submatrix corresponding to rows $k_1, k_2$ of the coefficient matrix of
Equation~\eqref{eq:SystemCongruence} has determinant
\begin{equation}
\label{eq:SystemCongruencePluecker}
    d_{i k_1} d_{j k_2} - d_{j k_1} d_{i k_2} = d_{i j} d_{k_1 k_2}.
\end{equation}
Thus,
\[
     \Delta_2
     =   \gcd\{d_{i k_1} d_{j k_2} - d_{j k_1} d_{i k_2} :
            k_1,k_2\neq i,j \}
     =  \gcd\{d_{ij}d_{k_1 k_2} : k_1,k_2\neq i,j \}
     =  d_{ij}\Delta_2^\prime,
\]
where $\Delta_2^\prime = \gcd\{d_{k_1 k_2} : k_1,k_2\neq i,j \}$. Then the number of solutions is given by
\[
     \gcd(\Delta_2/\Delta_1, d_{ij}) \gcd(\Delta_1, d_{ij})
        =   \frac{1}{\Delta_1} \gcd(d_{ij}\Delta'_2, d_{ij}\Delta_1) \gcd(\Delta_1, d_{ij})
        =   \frac{|d_{ij}|}{\Delta_1} \gcd(\Delta'_2,\Delta_1) \gcd(d_{ij},\Delta_1).
\]
Noting that $\gcd(\Delta_2^\prime, \Delta_1, d_{ij}) = 1$ as $A$ is faithful, $\gcd(\Delta_2^\prime,\Delta_1)$
and $\gcd(d_{ij},\Delta_1)$ are relatively prime, and we can write the number of solutions as
\[
     \frac{|d_{ij}|}{\Delta_1}  \gcd(\Delta_2^\prime,\Delta_1) \gcd(d_{ij},\Delta_1)
     =  \frac{|d_{ij}|}{\Delta_1} \gcd(d_{ij}\Delta_2^\prime,\Delta_1)
     =  \frac{|d_{ij}|}{\Delta_1} \gcd(\Delta_2,\Delta_1)
     =  |d_{ij}| \gcd(\Delta_2/ \Delta_1,1)=|d_{ij}|.
\]

If $A$ is not faithful so that $g > 1$, then we may apply the above result to conclude that there are
$|d_{ij}|/g$ pairs $(\eta,\nu)$ of $|d_{ij}|/g$th roots of unity such that
$\eta^{d_{ik}/g} \nu^{d_{jk}/g} = 1$ for all $k\neq i,j$. Considering the surjective homomorphism
$(\Z/d_{ij}\Z)^2\to(\Z/(d_{ij}/g)\Z)^2$ given by component-wise multiplication by $g$ completes the proof.
\end{proof}


\section{Computation of the Hilbert series}
\label{sec:HilbSer}

In this section, we give a formula for the Hilbert series $\Hilb_A^{\on}(t)$
of a representation $V_A$ of $\T^2$, analogous to the formula given in
\cite[Theorem~3.1]{HerbigSeaton}. We start with a formula for the completely generic case in Theorem~\ref{thrm:HilbGeneric}
which we then extend to the generic and degenerate case in Corollary~\ref{cor:HilbDegenerate}.


\subsection{A first formula}
\label{subsec:HilbGeneric}

Here we have the following.

\begin{theorem}
\label{thrm:HilbGeneric}
Let $n > 2$ and let $A\in\Z^{2\times n}$ be a faithful completely generic weight matrix. The Hilbert series
$\Hilb_A^{\on}(t)$ is given by
\begin{equation}
\label{eq:HilbGeneric}
    \sum_{\substack{i\neq j, \\ d_{ij} > 0}} \quad
        \sum\limits_{\substack{\xi^{d_{ij}}=1 \\ \zeta^{d_{ij}} = 1}}
        \frac{1}{d_{ij}^2
            \prod\limits_{k\neq i,j}
                \big(1 - \xi^{d_{ik}} \zeta^{d_{jk}}
                    t^{(d_{ij} + d_{ik} + d_{jk})/d_{ij} } \big)
                \big(1 - \xi^{-d_{ik}} \zeta^{-d_{jk}}
                    t^{(d_{ij} - d_{ik} - d_{jk})/d_{ij} } \big)  }.
\end{equation}
\end{theorem}
\begin{proof}
For $t=0$, the formula holds trivially, as we have $\Hilb_A^{\on}(0)=\dim_\R \R =1$. Thus, we may assume $t\ne 0$.
By the Molien-Weyl Theorem \cite[Section~4.6.1]{DerksenKemperBook}, the Hilbert series of the
off-shell invariants is given by the iterated integral over the torus $\T^2$
\begin{equation}
\label{eq:GenHilbIntegral}
    \frac{1}{(2\pi\sqrt{-1})^2} \int\limits_{\Sp^1} \int\limits_{\Sp^1}
        \frac{dz_1 dz_2}{z_1 z_2
            \prod\limits_{i=1}^n (1 - t z_1^{a_{1i}} z_2^{a_{2i}})(1 - t z_1^{-a_{1i}} z_2^{-a_{2i}})}.
\end{equation}
In order to compute this integral, we define $N = \prod_{i=1}^n a_{1i}$ and perform the substitution
$z_2 = w^N$ to yield
\[
    \frac{1}{(2\pi\sqrt{-1})^2} \int\limits_{\Sp^1} \int\limits_{\Sp^1}
        \frac{dz_1 dw}{z_1 w
            \prod\limits_{i=1}^n (1 - t z_1^{a_{1i}} w^{N a_{2i}})(1 - t z_1^{-a_{1i}} w^{- Na_{2i}})}.
\]
Assume $|t| < 1$ and $|w| = 1$ and define the integrand
\[
    F_{t,w}(z)  =   \frac{1}
        {zw \prod\limits_{i=1}^n (1 - t z^{a_{1i}} w^{N a_{2i}})(1 - t z^{-a_{1i}} w^{- Na_{2i}})}.
\]
We first consider the integral of $F_{t,w}(z)$ over $z \in \Sp^1$.

Note that as each $a_{1i} > 0$, we can express
\[
    F_{t,w}(z)  =   \frac{z^{-1 + \sum_{i=1}^n a_{1i}} }
        {w\prod\limits_{i=1}^n (1 - t z^{a_{1i}} w^{N a_{2i}})(z^{a_{1i}} - t w^{- Na_{2i}})}
\]
to see that $F_{t,w}(z)$ is holomorphic at $z = 0$. As $|t| < 1$ and $|w| = 1$, each of the factors
$(1 - t z^{a_{1i}} w^{N a_{2i}})$ is nonzero on the unit disk. Hence, the relevant poles are solutions
to $z^{a_{1i}} - t w^{- Na_{2i}} = 0$, of the form
$z = \eta t^{1/a_{1i}} w^{- Na_{2i}/a_{1i}}$ where $\eta$ is a fixed $a_{1i}$th root of unity.
Note that as $|\eta t^{1/a_{1i}} w^{- Na_{2i}/a_{1i}}| = |t|^{1/a_{1i}}$ and $A$ is completely generic, the poles are
distinct, i.e., each $i$ and $a_{1i}$th root of unity $\eta$ corresponds to a distinct pole.

Fix an $i$ and express
\[
    F_{t,w}(z)  =   \frac{z^{a_{1i} - 1}}
        {w(1 - t z^{a_{1i}} w^{N a_{2i}})(z^{a_{1i}} - t w^{- Na_{2i}})
        \prod\limits_{\substack{j=1 \\ j \neq i}}^n (1 - t z^{a_{1j}} w^{N a_{2j}})(1 - t z^{-a_{1j}} w^{- Na_{2j}})}.
\]
Fix an $a_{1i}$th root of unity $\eta_0$, expand the factor
\begin{align*}
    (z^{a_{1i}} - t w^{- Na_{2i}})
    &=      (z - \eta_0 t^{1/a_{1i}} w^{- Na_{2i}/a_{1i}})
            \prod\limits_{\substack{\eta^{a_{1i}}=1 \\ \eta\neq\eta_0}}
                (z - \eta t^{1/a_{1i}} w^{- Na_{2i}/a_{1i}}),
\end{align*}
and note that
\begin{align*}
    \prod\limits_{\substack{\eta^{a_{1i}}=1 \\ \eta\neq\eta_0}}
        (\eta_0 t^{1/a_{1i}} w^{- Na_{2i}/a_{1i}} - \eta t^{1/a_{1i}} w^{- Na_{2i}/a_{1i}})
    &=      \big(\eta_0 t^{1/a_{1i}} w^{- Na_{2i}/a_{1i}}\big)^{a_{1i} - 1}
                \prod\limits_{\substack{\eta^{a_{1i}}=1 \\ \eta\neq 1}} (1 - \eta )
    \\&=    a_{1i} \big(\eta_0 t^{1/a_{1i}} w^{- Na_{2i}/a_{1i}}\big)^{a_{1i} - 1}.
\end{align*}
Therefore, the residue of $F_{t,w}(z)$ at $z = \eta_0 t^{1/a_{1i}} w^{- Na_{2i}/a_{1i}}$ is given by
\[
    \frac{1}
        {wa_{1i}(1 - t^2)
            \prod\limits_{j\neq i} (1 - \eta_0^{a_{1j}} t^{1 + a_{1j}/a_{1i}} w^{N (a_{2j} - a_{1j} a_{2i}/a_{1i})} )
                (1 - \eta_0^{-a_{1j}} t^{1 - a_{1j}/a_{1i}} w^{- N(a_{2j} - a_{1j} a_{2i}/a_{1i})} )}.
\]
Letting $q_i = \prod_{j\neq i} a_{1j} = N/a_{1i}$, we can express this residue as
\[
    \frac{1}
        {wa_{1i}(1 - t^2)
            \prod\limits_{j\neq i} (1 - \eta_0^{a_{1j}} t^{1 + a_{1j}/a_{1i}} w^{q_i d_{ij}} )
                (1 - \eta_0^{-a_{1j}} t^{1 - a_{1j}/a_{1i}} w^{- q_i d_{ij}} )}.
\]
Summing residues over each choice of $i$ and corresponding roots of unity $\eta$, the outer integral is given by
\begin{equation}
\label{eq:GenHilbFirstIntegral}
    (2\pi\sqrt{-1}) \sum_{i=1}^n \sum_{\eta^{a_{1i}} = 1}
    \frac{1}{wa_{1i} (1 - t^2)
        \prod\limits_{j\neq i}
            (1 - \eta^{a_{1j}} t^{1 + a_{1j}/a_{1i}} w^{q_i d_{ij}})
            (1 - \eta^{-a_{1j}} t^{1 - a_{1j}/a_{1i}} w^{- q_i d_{ij}})}.
\end{equation}
Note that formally $t^{1/a_{1i}}$ is well-defined only after fixing a branch
of the logarithm. However, Expression~\eqref{eq:GenHilbFirstIntegral} sums over all the distinct $a_{1i}$th
roots of $t$ and is therefore well-defined independently of the chosen branch.
We set
\begin{equation}
\label{eq:GenHilbDefBeta}
    \beta_{ij}(\eta,w) =
        (1 - \eta^{a_{1j}} t^{1 + a_{1j}/a_{1i}} w^{q_i d_{ij}})
        (1 - \eta^{-a_{1j}} t^{1 - a_{1j}/a_{1i}} w^{- q_i d_{ij}})
\end{equation}
and then can express \eqref{eq:GenHilbFirstIntegral} succinctly as
\[
    (2\pi\sqrt{-1}) \sum_{i=1}^n \sum_{\eta^{a_{1i}} = 1} \frac{1}
        {w a_{1i} (1 - t^2) \prod\limits_{j\neq i} \beta_{ij}(\eta,w)}.
\]
Note that for fixed $t$, this function is rational in $w$.

Fix a value of $i$ and an $a_{1i}$th root of unity $\eta$. We now proceed with the integral
of the corresponding term of Expression~\eqref{eq:GenHilbFirstIntegral} with respect to $w$.

For each $j$, the first factor of $\beta_{ij}(\eta,w)$ has a root on the unit disk if and only if $d_{ij} < 0$,
while the second factor has a root if and only if $1 - a_{1j}/a_{1i}$ and $d_{ij}$ have the same sign.
Note also that $1 - a_{1j}/a_{1i} = 0$ is impossible as $A$ is completely generic (and hence in standard form).

We consider the roots of the first factor of $\beta_{ij}(\eta,w)$.
Assume $d_{ij} < 0$. Express $1/\big(wa_{1i}(1 - t^2) \prod\limits_{j\neq i} \beta_{ij}(\eta,w)\big)$ as
\[
    \frac{w^{- q_i d_{ij} - 1}}
        {a_{1i}(1 - t^2)(w^{- q_i d_{ij}} - \eta^{a_{1j}} t^{1 + a_{1j} /a_{1i}})
        (1 - \eta^{-a_{1j}} t^{1 - a_{1j}/a_{1i}} w^{-q_i d_{ij}}) \prod\limits_{k\neq i,j} \beta_{ik}(\eta,w)  },
\]
and then the factor $(w^{- q_i d_{ij}} - \eta^{a_{1j}} t^{1 + a_{1j} /a_{1i}})$ in the denominator can be expressed as
\[
    (w - \nu_0 \eta^{- a_{1j} /(q_i d_{ij})} t^{- (1 + a_{1j}/a_{1i})/(q_i d_{ij})})
    \prod\limits_{\substack{ \nu^{-q_i d_{ij}} = 1 \\ \nu \neq \nu_0 }}
        (w - \nu \eta^{- a_{1j} / (q_i d_{ij})} t^{- (1 + a_{1j}/a_{1i})/(q_i d_{ij})})
\]
where $\nu_0$ is a $-q_i d_{ij}$th root of unity. Hence, poles corresponding to the vanishing of the first factor
of $\beta_{ij}(\eta,w)$ are of the form
$\tau_1(i,j,\eta,\nu_0) := \nu_0 \eta^{- a_{1j} /(q_i d_{ij})} t^{- (1 + a_{1j}/a_{1i})/(q_i d_{ij})}$.
Note that $|\tau_1(i,j,\eta,\nu_0)| = |t|^{- (1 + a_{1j}/a_{1i})/(q_i d_{ij})} = |t|^{- (a_{1i} + a_{1j})/(N d_{ij})}$,
and, for $j\neq k$ (and both distinct from $i$), we have
$(a_{1i} + a_{1j})/d_{ij} = (a_{1i} + a_{1k})/d_{ik}$ if and only if $d_{ij} - d_{ik} - d_{jk} = 0$.
That is, the hypothesis that $A$ is completely generic implies that the poles
$\{ \tau_1(i,j,\eta,\nu_0) : j\neq i, \nu_0^{-q_i d_{ij}} = 1 \}$ are distinct.

The residue at $\tau_1 = \tau_1(i,j,\eta,\nu_0)$ is given by
\begin{align*}
    &\frac{\tau_1^{- q_i d_{ij} - 1}}{a_{1i}(1 - t^2)
        \prod\limits_{\substack{ \nu^{-q_i d_{ij}} = 1 \\ \nu \neq \nu_0 }}
            (\tau_1 - \nu \eta^{- a_{1j} / (q_i d_{ij})} t^{- (1 + a_{1j}/a_{1i})/(q_i d_{ij})})
            (1 - \eta^{-a_{1j}} t^{1 - a_{1j} / a_{1i}} \tau_1^{-q_i d_{ij}})
            \prod\limits_{k\neq i,j} \beta_{ik}(\eta,\tau_1)    }
    \\&\quad\quad\quad=
    \frac{-1}
        {N d_{ij} (1 - t^2)^2
            \prod\limits_{k\neq i,j} \beta_{ik}(\eta,\tau_1)  }.
\end{align*}
Substituting $\tau_1 = \nu_0 \eta^{- a_{1j} /(q_i d_{ij})} t^{- (1 + a_{1j}/a_{1i})/(q_i d_{ij})}$
into the definition of $\beta_{ik}$ in Equation~\eqref{eq:GenHilbDefBeta}, we have
\begin{align*}
    \beta_{ik}(\eta,\tau_1)
    &=   \Big(1 - \nu_0^{q_i d_{ik}} \eta^{a_{1k} - a_{1j}d_{ik} / d_{ij} }
            t^{(d_{ij}(a_{1i} + a_{1k}) - d_{ik}(a_{1i} + a_{1j}))/(d_{ij}a_{1i}) }\Big)\cdot
    \\&\quad\quad\quad
        \Big(1 - \nu_0^{- q_i d_{ik}} \eta^{-a_{1k} + a_{1j}d_{ik} / d_{ij} }
            t^{(d_{ij}(a_{1i} - a_{1k}) + d_{ik}(a_{1i} + a_{1j}))/(d_{ij}a_{1i}) }\Big).
\end{align*}
Simplifying the exponents using the identity $a_{1i}d_{jk} + a_{1j}d_{ki} + a_{1k}d_{ij} = 0$,
we express this residue as
\[
    R_1 =
    \frac{-1}
        {N d_{ij} (1 - t^2)^2
            \prod\limits_{k\neq i,j}
                \big(1 -
                    \nu_0^{q_i d_{ik}}
                    \eta^{- a_{1i} d_{jk} /d_{ij} }
                    t^{(d_{ij} - d_{ik} - d_{jk})/d_{ij} }
                    \big)
                \big(1 -
                    \nu_0^{- q_i d_{ik}}
                    \eta^{a_{1i} d_{jk}/ d_{ij} }
                    t^{(d_{ij} + d_{ik} + d_{jk})/d_{ij} }
                \big)  }.
\]
Let $\zeta = \eta^{-a_{1i}/d_{ij}}$ and $\xi = \nu_0^{q_i}$, and then
\[
    \sum_{\substack{\eta^{a_{1i}} = 1 \\ \nu_0^{-q_i d_{ij}} = 1}} R_1
    =
    \sum_{\substack{\zeta^{-d_{ij}} = 1 \\ \xi^{-d_{ij}} = 1 }} R_1^\prime
\]
where
\begin{equation}
\label{eq:GenHilbSecondIntegralRes1}
    R_1^\prime
    =
    \frac{1}
        {d_{ij}^2 (1 - t^2)^2
            \prod\limits_{k\neq i,j}
                \big(1 - \xi^{d_{ik}}
                    \zeta^{d_{jk}}
                    t^{(d_{ij} - d_{ik} - d_{jk})/d_{ij} }
                    \big)
                \big(1 -
                    \xi^{- d_{ik}}
                    \zeta^{- d_{jk}}
                    t^{(d_{ij} + d_{ik} + d_{jk})/d_{ij} }
                \big)  }.
\end{equation}
Once again, the formalism of choosing a fixed branch of the logarithm for the substitution $\zeta = \eta^{-a_{1i}/d_{ij}}$
was replaced here by the process of averaging over distinct roots of unity.

We now turn to roots of the second factor of $\beta_{ij}(\eta,w)$. First assume
$d_{ij} > 0$ and $1 - a_{1j}/a_{1i} > 0$, i.e., $a_{1i} > a_{1j}$. We express the integrand
$1/\big(w a_{1i}(1 - t^2) \prod\limits_{j\neq i} \beta_{ij}(\eta,w)\big)$ as
\[
    \frac{w^{q_i d_{ij} - 1}}
        {a_{1i}(1 - t^2)(1 - \eta^{a_{1j}} t^{1 + a_{1j}/a_{1i}} w^{q_i d_{ij}})
            (w^{q_i d_{ij}} - \eta^{-a_{1j}} t^{1 - a_{1j}/a_{1i}}) \prod\limits_{k\neq i,j} \beta_{ik}(\eta,w) },
\]
and factor $(w^{q_i d_{ij}} - \eta^{-a_{1j}} t^{1 - a_{1j}/a_{1i}})$ into
\[
    (w - \nu_0\eta^{-a_{1j}/(q_i d_{ij})}t^{(1 - a_{1j}/a_{1i})/(q_i d_{ij})})
                \prod\limits_{\substack{\nu^{q_i d_{ij}=1} \\ \nu\neq\nu_0}}
            (w - \nu \eta^{-a_{1j}/(q_i d_{ij})}t^{(1 - a_{1j}/a_{1i})/(q_i d_{ij})}),
\]
where $\nu_0$ is a $q_i d_{ij}$th root of unity. The corresponding simple poles occur when $w$
is equal to $\tau_2(i,j,\eta,\nu_0):= \nu_0\eta^{-a_{1j}/(q_i d_{ij})} t^{(1 - a_{1j}/a_{1i})/(q_i d_{ij})}$.
As $|\tau_2(i,j,\eta,\nu_0)| = |t|^{(a_{1i} - a_{1j})/(N d_{ij})}$,
$(a_{1i} - a_{1j})/d_{ij} = (a_{1i} - a_{1k})/d_{ik}$ if and only if $d_{ij} - d_{ik}  + d_{jk} = 0$, and
$(a_{1i} - a_{1j})/d_{ij} = -(a_{1i} + a_{1k})/d_{ik}$ if and only if $d_{ij} + d_{ik} - d_{jk} = 0$,
the fact that $A$ is completely generic implies that these poles are all distinct, and are distinct from the poles $\tau_1$ above.
A computation similar to the previous case expresses the residue as
\[
    R_2 =
    \frac{1}{N d_{ij} (1 - t^2)^2
        \prod\limits_{k\neq i,j}
            \big(1 - \nu_0^{q_i d_{ik}} \eta^{-a_{1i}d_{jk}/d_{ij}}
                t^{(d_{ij} + d_{ik} - d_{jk})/ d_{ij}}\big)
            \big(1 - \nu_0^{-q_i d_{ik}} \eta^{a_{1i}d_{jk}/d_{ij}}
                t^{(d_{ij} - d_{ik}  + d_{jk})/ d_{ij} } \big) }.
\]
Applying the same substitutions as in the previous case, we have
\[
    \sum_{\substack{\eta^{a_{1i}} = 1 \\ \nu_0^{q_i d_{ij}} = 1}} R_2
    =
    \sum_{\substack{\zeta^{d_{ij}} = 1 \\ \xi^{d_{ij}} = 1 }} R_2^\prime
\]
where
\begin{equation}
\label{eq:GenHilbSecondIntegralRes2}
    R_2^\prime =
    \frac{1}{d_{ij}^2 (1 - t^2)^2
        \prod\limits_{k\neq i,j}
            \big(1 - \xi^{d_{ik}} \zeta^{d_{jk}}
                t^{(d_{ij} + d_{ik} - d_{jk})/ d_{ij}}\big)
            \big(1 - \xi^{- d_{ik}} \zeta^{-d_{jk}}
                t^{(d_{ij} - d_{ik}  + d_{jk})/ d_{ij} } \big) }.
\end{equation}
If $d_{ij} < 0$ and $1 - a_{1j}/a_{1i} < 0$, i.e., $a_{1i} < a_{1j}$, a practically identical computation identifies
again simple poles of the form $\tau_2(i,j,\eta,\nu_0) = \nu_0 \eta^{- a_{1j} / (q_i d_{ij})} t^{(1 - a_{1j} / a_{1i})/ (q_i d_{ij})}$
with residue $R_2$, while our standard substitution results in the slightly modified equation

\[
    \sum_{\substack{\eta^{a_{1i}} = 1 \\ \nu_0^{q_i d_{ij}} = 1}} R_2
    =
    \sum_{\substack{\zeta^{d_{ij}} = 1 \\ \xi^{d_{ij}} = 1 }} -R_2^\prime.
\]
Combining these computations, it follows that the integral in Equation~\eqref{eq:GenHilbIntegral}
is given by
\[
    (2\pi\sqrt{-1})^2 \sum\limits_{i=1}^n \left(
        \sum\limits_{\substack{ j = 1 \\ j \neq i, \: d_{ij} < 0}}^n \quad
        \sum_{\substack{\zeta^{-d_{ij}} = 1 \\ \xi^{-d_{ij}} = 1 }} R_1^\prime
        + \sum\limits_{\substack{j = 1 \\  j \neq i, \: d_{ij} > 0 \\ a_{1i} > a_{1j}}}^n \quad
        \sum_{\substack{\zeta^{d_{ij}} = 1 \\ \xi^{d_{ij}} = 1 }} R_2^\prime
        + \sum\limits_{\substack{j = 1 \\  j \neq i, \: d_{ij} < 0 \\ a_{1i} < a_{1j}}}^n \quad
        \sum_{\substack{\zeta^{d_{ij}} = 1 \\ \xi^{d_{ij}} = 1 }} -R_2^\prime
    \right).
\]
Switching the roles of $i$ and $j$ as well as substituting $\zeta \mapsto \xi^{-1}$ and $\xi \mapsto \zeta^{-1}$ in the third sum yields the negative of
the second sum, leaving only the first sum. Then switching the roles of $i$ and $j$ as well as $\zeta$ and $\xi$
in the first sum, the off-shell Hilbert series is given by
\[
    \sum_{\substack{i\neq j, \\ d_{ij} > 0}} \quad
        \sum\limits_{\substack{\xi^{d_{ij}}=1 \\ \zeta^{d_{ij}} = 1}}
        \frac{1}{d_{ij}^2 (1 - t^2)^2
            \prod\limits_{k\neq i,j}
                \big(1 - \xi^{d_{ik}} \zeta^{d_{jk}}
                    t^{(d_{ij} + d_{ik} + d_{jk})/d_{ij} } \big)
                \big(1 - \xi^{-d_{ik}} \zeta^{-d_{jk}}
                    t^{(d_{ij} - d_{ik} - d_{jk})/d_{ij} } \big)  }.
\]
Applying Equation~\eqref{eq:OnOffHilb} (\cite[Lemma~2.1]{HerbigSeaton}),
$\Hilb_A^{\on}(t)$ is the product of $(1 - t^2)^2$ and the off-shell Hilbert series, completing the proof.
\end{proof}


\subsection{Analytic continuation}
\label{subsec:HilbDegenerate}

Revisiting Theorem~\ref{thrm:HilbGeneric}, there is no particular reason why the final expression
\eqref{eq:HilbGeneric} should depend on the additional condition $d_{ij}+d_{jk}+d_{ki}\ne 0$ for every distinct $i,j,k$.
Yet again, if $A$ is degenerate, then there are distinct $i,j,k$ such that $d_{ij} - d_{ik} - d_{jk} = 0$,
and Expression \eqref{eq:HilbGeneric} fails to be well-defined due to division by zero in the case of $\xi =\zeta =1$.
Specifically, as $(a_{1i} + a_{1j})/d_{ij} =  (a_{1i} + a_{1k})/d_{ik}$, the poles
identified in the computation in the proof of Theorem~\ref{thrm:HilbGeneric} are not distinct and hence are not simple poles.
Hence, the computation does not apply. Nevertheless, the result of Theorem~\ref{thrm:HilbGeneric} can be extended to the case
of general generic and degenerate $A$ with the help of analytic continuation.

\begin{lemma}
\label{lem:RemovSing}
Let $C$ be a simple closed curve, let $f(z)$ be a continuous function on $C$,
and let $\tau$ be interior to $C$. Then
\[
    \lim\limits_{(\tau_1,\ldots,\tau_m)\to(\tau,\ldots,\tau)}
    \int\limits_C \frac{f(z)\,dz}{\prod\limits_{i=1}^m (z - \tau_i)}
    =
    \int\limits_C \frac{f(z)\,dz}{(z - \tau)^m}.
\]
\end{lemma}
\begin{proof}
Let $g(z,\tau_1,\ldots,\tau_m) = f(z)/\big(\prod_{i=1}^m (z - \tau_i)\big)$ denote the integrand
as a function of $z$ and the $\tau_i$. Let $D$ denote a closed $\epsilon$-ball about $\tau$ that is contained in
the interior of $C$, and then $g(z,\tau_1,\ldots,\tau_m)$ is continuous on the compact set $C\times D^m$.
It follows that $g(z,\tau_1,\ldots,\tau_m)$ is bounded by a constant on this set, and the result follows
from an application of the dominated convergence theorem.
\end{proof}

If $f(z) = g(z)/h(z)$ is a rational function, where $h(z)$ has no zeros on or inside $C$,
then we can understand the limit in Lemma~\ref{lem:RemovSing} as
follows. Choosing the $\tau_i$ distinct inside $C$, we have
\[
    \frac{1}{2\pi\sqrt{-1}}\int\limits_C \frac{g(z)\,dz}{h(z)\prod\limits_{i=1}^m (z - \tau_i)}
    =
    \sum\limits_{i=1}^m \frac{g(\tau_i)}{h(\tau_i)\prod\limits_{\substack{j=1\\ j\neq i}}^m (\tau_i - \tau_j)},
\]
which we rewrite as a single rational fraction $\frac{p(\tau_1,\ldots,\tau_m)}{q(\tau_1,\ldots,\tau_m)}$ with common denominator
\[
    q(\tau_1,\ldots,\tau_m) = \prod_{i=1}^m h(\tau_i)\prod_{1\leq j < k \leq m} (\tau_k - \tau_j).
\]
Note that by definition $\frac{p(\tau_1,\ldots,\tau_m)}{q(\tau_1,\ldots,\tau_m)}$ is symmetric in the $\tau_i$ while
$q(\tau_1,\ldots,\tau_m)$ is alternating. Therefore, the numerator $p(\tau_1,\ldots,\tau_m)$ is an alternating polynomial
in the $\tau_i$ and hence divisible by the Vandermonde determinant $\prod_{1\leq j < k \leq m} (\tau_k - \tau_j)$, i.e.,
\[
    p(\tau_1,\ldots,\tau_m) = s(\tau_1,\ldots,\tau_m)\prod\limits_{1\leq j < k \leq m} (\tau_k - \tau_j)
\]
for some symmetric polynomial $s$ in the $\tau_i$. Therefore, the singularities at $\tau_i = \tau_j$ are removable,
and we can express the integral as
\[
    \frac{p(\tau_1,\ldots,\tau_m)}{q(\tau_1,\ldots,\tau_m)}
    =
    \frac{s(\tau_1,\ldots,\tau_m)}{\prod_{i=1}^m h(\tau_i)}.
\]

In the proof of Theorem~\ref{thrm:HilbGeneric}, each of the integrands of the iterated integral is a rational
function. Hence, using Lemma~\ref{lem:RemovSing}, we can perturb the poles with multiplicity and apply the
same computation. In more detail, in the integral with respect to $z$, if multiple poles that are solutions of
factors of the form $z^{a_{1i}} - t w^{- Na_{2i}} = 0$ coincide, we may perturb these factors by replacing $t$
with a separate variable $t_i$ in each, resulting in a rational function in $z$ that has only
simple poles in the unit disk. We then compute the integral of this function and then take the limit as the
$t_i \to t$. Similarly, in the second integral with respect to $w$, we may similarly perturb $t$ in factors
associated to common poles, compute the integral of the resulting
rational function with only simple poles, and then take the limit as these perturbed variables return to $t$.
In order to state the resulting formula, we express the perturbed variables $t_i$ in the form $t_i = t^{p_i}$
where $p_i$ is near $1$ and the exponent is defined using a fixed branch of the logarithm that is defined
on a neighborhood of $t$. See \cite[page~52]{HerbigSeaton} for more details on this approach in a similar
computation. Then, taking advantage of the continuity of power functions within the domain of the
fixed branch of $\log$, we have the following.

\begin{corollary}
\label{cor:HilbDegenerate}
Let $n > 2$ and let $A\in\Z^{2\times n}$ be a faithful weight matrix in standard form.
The Hilbert series $\Hilb_A^{\on}(t)$ is given by
\begin{equation}
\label{eq:HilbDegenerate}
    \lim\limits_{X\to A}
    \sum_{\substack{i\neq j, \\ d_{ij} > 0}} \quad
        \sum\limits_{\substack{\xi^{d_{ij}}=1 \\ \zeta^{d_{ij}} = 1}}
        \frac{1}{d_{ij}^2
            \prod\limits_{k\neq i,j}
                \big(1 - \xi^{d_{ik}} \zeta^{d_{jk}}
                    t^{(c_{ij} + c_{ik} + c_{jk})/c_{ij} } \big)
                \big(1 - \xi^{-d_{ik}} \zeta^{-d_{jk}}
                    t^{(c_{ij} - c_{ik} - c_{jk})/c_{ij} } \big)  }
\end{equation}
where the $x_{ij}$ are real parameters approximating the $a_{ij}$, $X = (x_{ij})$,
$c_{ij} = x_{1i} x_{2j} - x_{2i} x_{1j}$, and the power functions are computed using a fixed branch of the logarithm.
In particular,
\begin{equation}
\label{eq:HilbGeneric2}
    \Hilb_A^{\on}(t)=\sum_{\substack{i\neq j, \\ d_{ij} > 0}} \quad
        \sum\limits_{\substack{\xi^{d_{ij}}=1 \\ \zeta^{d_{ij}} = 1}}
        \frac{1}{d_{ij}^2
            \prod\limits_{k\neq i,j}
                \big(1 - \xi^{d_{ik}} \zeta^{d_{jk}}
                    t^{(d_{ij} + d_{ik} + d_{jk})/d_{ij} } \big)
                \big(1 - \xi^{-d_{ik}} \zeta^{-d_{jk}}
                    t^{(d_{ij} - d_{ik} - d_{jk})/d_{ij} } \big)  }.
\end{equation}
holds for any faithful generic weight matrix $A$.
\end{corollary}

\begin{remark}
\label{rem:OffShellHilb}
Using Equation~\eqref{eq:OnOffHilb}, Corollary~\ref{cor:HilbDegenerate} also yields a formula for
the Hilbert series $\Hilb_A^{\off}(t)$ of the off-shell invariants of the cotangent-lifted representation
associated to $A$, i.e., the usual real invariants of the representation with weight matrix $A$, or
equivalently the complex invariants of the representation with weight matrix $(A|-A)$.
Explicitly,
\[
    \Hilb_A^{\off}(t) = \frac{\Hilb_A^{\on}(t)}{(1 - t^2)^2}.
\]
\end{remark}


\subsection{An algorithm to compute the Hilbert series}
\label{subsec:Algorithm}

As in the case of circle quotients treated in \cite{HerbigSeaton}, Equation~\eqref{eq:HilbGeneric2} indicates
an algorithm to compute the Hilbert series $\Hilb_A^{\on}(t)$ in the case of a generic weight matrix that we
now describe. First, for a ring $R$ containing $\Q$, let $R(\!(t)\!)$ denote the ring of formal Laurent polynomials
in $t$ over $R$, and for $d\in\N$, define the operator $U_{d,t}\co R(\!(t)\!)\to R(\!(t)\!)$ by
\begin{equation}
\label{eq:UOpPowerSeries}
    U_{d,t} \left(\sum\limits_{m\in\Z} F_m t^m \right) = \sum\limits_{m\in\Z} F_{md} t^m.
\end{equation}
This operator generalizes that defined in \cite[Section~4]{HerbigSeaton} and has similar properties. Specifically,
for $F(t) = \sum_{m\in\Z} F_m t^m$,
\begin{equation}
\label{eq:UOpRootUnity}
    U_{d,t}\big(F(t)\big) = \frac{1}{d}\sum\limits_{\zeta^d = 1} F\big( \zeta\sqrt[d]{t} \big).
\end{equation}
The idea behind the algorithm is to interpret Equation~\eqref{eq:HilbGeneric2} in terms of composing operators
of the form $U_{d,t}$. Specifically, we can write
\[
    \Hilb_A^{\on}(t)
        =   \sum_{\substack{i\neq j, \\ d_{ij} > 0}}
                \big( U_{d_{ij},s} \circ U_{d_{ij},t} \big)(\Phi_{ij}(s,t))\Big|_{s=t}
\]
where
\begin{equation}
\label{eq:AlgPhiDef}
    \Phi_{ij}(s,t)
        :=  \frac{1}{d_{ij}^2 \prod\limits_{k\neq i,j}
                \big(1 - s^{ d_{ik} } t^{ d_{ij} + d_{jk} } \big)
                \big(1 - s^{ -d_{ik} } t^{ d_{ij} - d_{jk} } \big)  }.
\end{equation}
Using Equation~\eqref{eq:UOpRootUnity}, note that if $F(s,t) = P(s,t)/Q(s,t)$ where $P$ and $Q$ are
polynomials in $t$ with coefficients in $\Q[s,s^{-1}]$,
$\delta = \operatorname{deg}_t P(s,t) - \operatorname{deg}_t Q(s,t)$
where $\operatorname{deg}_t$ means the degree as a polynomial in $t$,
then $U_{d,t}\big(F(s,t)\big) = P_d(s,t)/Q_d(s,t)$ where $P_d$ and $Q_d$ are
polynomials in $t$ with coefficients in
$\Q[s,s^{-1}]$ such that
\begin{equation}
\label{eq:UOpDegreeBound}
    \operatorname{deg}_t Q_d(s,t) = \operatorname{deg}_t Q(s,t)\qquad \mbox{ and}\qquad \operatorname{deg}_t P_d(s,t) - \operatorname{deg}_t Q_d(s,t) \le \lfloor \delta/d \rfloor.
\end{equation}

With these observations, the algorithm is as follows.

Given a generic weight matrix $A$, fix $i,j$ such that $d_{ij} > 0$ and define the function
$\Phi(s,t) = \Phi_{ij}(s,t)$ as in Equation~\eqref{eq:AlgPhiDef}. Then do the following:
\begin{enumerate}
\item   For each factor $(1 - s^p t^q)$ in the denominator such that $q < 0$, multiply the numerator
        and denominator by the monomial $-s^{-p} t^{-q}$ so that all powers of $t$ in the denominator
        are nonnegative. Let $P(s,t)$ and $Q(s,t)$ denote the resulting numerator and denominator,
        respectively. Define $\delta = \operatorname{deg}_t P(s,t) - \operatorname{deg}_t Q(s,t)$.

\item   Define the function $Q_1(s,t)$ by replacing each of the factors of the form $(1 - s^p t^q)$
        in $Q(s,t)$ via the rule
        \[
            (1 - s^p t^q)   \longmapsto
                \big(1 - s^{p d_{ij}/\gcd(d_{ij},q)} t^{q/\gcd(d_{ij},q)} \big)^{\gcd(d_{ij},q)}.
        \]
        Then $Q_1(s,t)$ is the denominator of $U_{d_{ij},t}\big(P(s,t)/Q(s,t)\big)$.

\item   To compute the numerator $P_1(s,t)$ of $U_{d_{ij},t}\big(P(s,t)/Q(s,t)\big)$,
        first compute the Taylor series of $P(s,t)/Q(s,t)$ with respect to $t$ at $t = 0$ up to degree
        $d_{ij}\big( \lfloor \delta/d_{ij} \rfloor + \operatorname{deg}_t Q_1(s,t) \big)$.
        Apply $U_{d_{ij},t}$ to this Taylor series using the description in
        Equation~\eqref{eq:UOpPowerSeries}, multiply the output series by $Q_1(s,t)$,
        and delete all terms with $\operatorname{deg}_t$ larger than
        $\lfloor \delta/d_{ij} \rfloor + \operatorname{deg}_t Q_1(s,t)$.
        Call the result $P_1(s,t)$, and then
        $U_{d_{ij},t}\big(P(s,t)/Q(s,t)\big) = P_1(s,t)/Q_1(s,t)$.

\item   For each factor $(1 - s^p t^q)$ in the denominator of
        $P_1(s,t)/Q_1(s,t)$ such that $p < 0$, multiply the numerator
        and denominator by the monomial $-s^{-p} t^{-q}$ so that all powers of $s$ in the denominator
        are nonnegative. Let $P_2(s,t)$ and $Q_2(s,t)$ denote the resulting numerator
        and denominator, respectively. Define
        $\delta^\prime = \operatorname{deg}_s P_2(s,t) - \operatorname{deg}_s Q_2(s,t)$.

\item   Define the function $Q_3(s,t)$ by replacing each of the factors of the form $(1 - s^p t^q)$
        in $Q_2(s,t)$ via the rule
        \[
            (1 - s^p t^q)   \longmapsto
                \big(1 - s^{p/\gcd(d_{ij},p)} t^{q d_{ij}/\gcd(d_{ij},p)} \big)^{\gcd(d_{ij},p)}.
        \]
        Then $Q_3(s,t)$ is the denominator of
        $U_{d_{ij},s}\big(P_2(s,t)/Q_2(s,t)\big)$.

\item   To compute the numerator $P_3(s,t)$ of $U_{d_{ij},s}\big(P_2(s,t)/Q_2(s,t)\big)$,
        first compute the Taylor series of $P_2(s,t)/Q_2(s,t)$ with respect to $s$ at $s = 0$ up to degree
        $d_{ij}\big( \lfloor \delta^\prime/d_{ij} \rfloor + \operatorname{deg}_s Q_3(s,t) \big)$.
        Apply $U_{d_{ij},s}$ to the result using the description in
        Equation~\eqref{eq:UOpPowerSeries}, multiply the output by $Q_3(s,t)$,
        and delete all terms with $\operatorname{deg}_s$ larger than
        $\lfloor \delta^\prime/d_{ij} \rfloor + \operatorname{deg}_s Q_3(s,t)$.
        The result is $P_3(s,t)$, and
        $U_{d_{ij},s}\big(P_2(s,t)/Q_2(s,t)\big) = P_3(s,t)/Q_3(s,t)$.
\end{enumerate}
Apply the above process for each $i,j$ such that $d_{ij} > 0$, sum each of the resulting terms
$P_3(s,t)/Q_3(s,t)$, and substitute $s = t$ in the sum. The result is $\Hilb_A^{\on}(t)$.

This algorithm has been implemented on \emph{Mathematica} \cite{Mathematica} and is available
from the authors upon request. It does not perform particularly well. The largest bottleneck appears
to be the computation of Taylor series expansions; even for $2\times 4$ weight matrices with
single-digit entries, the algorithm can require series expansions up to degrees in the hundreds,
which are computationally very expensive. It can handle many $2\times 3$ and some
$2\times 4$ examples. However, it does not perform better than computing the off-shell invariants
using the package \emph{Normaliz} \cite{Normaliz} for \emph{Macaulay2} \cite{M2} and using the resulting
description to compute the Hilbert series, and this latter method has often been more successful.
As an example, in the case of weight matrix
\[
    A = \begin{pmatrix} 1 & 2 & 3 & 4 & 5 \\ 0 & 1 & 2 & 2 & 1 \end{pmatrix},
\]
the invariants and Hilbert series were computed using \emph{Normaliz} and \emph{Macaulay2}
in under four hours on a computer with one core and 5GB RAM, while the algorithm described here
ran out of memory on a machine with 16GB RAM. The Hilbert series in this case is given by
\begin{multline*}
    \frac{1}{(1-t^3)(1-t^4)(1-t^9)(1-t^{10})(1-t^{11})(1-t^{15})}
    \big(1 + 3t^2 + 3t^3 + 7t^4 + 11t^5 + 19t^6 + 31t^7 + 47t^8 + 68t^9
        \\  + 92t^{10} + 121t^{11} + 153t^{12} + 188t^{13} + 232t^{14} + 273t^{15} + 318t^{16} + 359t^{17}
            + 393t^{18} + 426t^{19}
        \\  + 454t^{20} + 475t^{21} + 491t^{22} + 496t^{23} + 491t^{24} + 475t^{25} + 454t^{26} + 426t^{27}
            + 393t^{28} + 359t^{29}
        \\  + 318t^{30} + 273t^{31} + 232t^{32} + 188t^{33} + 153t^{34} + 121t^{35} + 92t^{36} + 68t^{37}
            + 47t^{38} + 31t^{39}
        \\  + 19t^{40} + 11t^{41} + 7t^{42} + 3t^{43} + 3t^{44} + t^{46}\big).
\end{multline*}


\section{Computation of the Laurent coefficients}
\label{sec:Gammas}

Let $A \in\Z^{2\times n}$ be a faithful weight matrix in standard form with $n > 2$. As in Section~\ref{sec:HilbSer},
we let $d_{ij}$ denote the $2\times 2$ minor associated to columns $i$ and $j$. If $A$ is degenerate, we approximate
the $a_{ij}$ with real parameters $x_{ij}$ and let $c_{ij} = x_{1i} x_{2j} - x_{2i} x_{1j}$ to assume that
$c_{ij} + c_{ik} + c_{jk} \neq 0$ for each distinct $i,j,k$. Let $X = (x_{ij})$, let
\begin{equation}
\label{eq:HilbTerm}
    H_{X,i,j,\xi,\zeta}(t)  = \frac{1}{c_{ij}d_{ij}
            \prod\limits_{k\neq i,j}
                \big(1 - \xi^{d_{ik}} \zeta^{d_{jk}}
                    t^{(c_{ij} + c_{ik} + c_{jk})/c_{ij} } \big)
                \big(1 - \xi^{-d_{ik}} \zeta^{-d_{jk}}
                    t^{(c_{ij} - c_{ik} - c_{jk})/c_{ij} } \big)  },
\end{equation}
and assume throughout this section that the power functions are defined using a fixed branch of $\log t$
such that $\log 1 = 0$. Let
\begin{equation}
\label{eq:HilbSumTerms}
    H_X(t) = \sum_{\substack{i\neq j, \\ d_{ij} > 0}} \quad
            \sum\limits_{\substack{\xi^{d_{ij}}=1 \\ \zeta^{d_{ij}} = 1}} H_{X,i,j,\xi,\zeta}(t),
\end{equation}
so that a minor adaptation of Equation~\eqref{eq:HilbDegenerate} (by setting one instance of $d_{ij}$ equal to $c_{ij}$
in each term) can be expressed as
\begin{equation}
\label{eq:HilbSumTermsZ}
\Hilb_A^{\on}(t) = \lim_{X\to A} H_X(t).
\end{equation}

In this section, we consider the Laurent expansion
\begin{equation}
\label{eq:GamDef}
    \Hilb_A^{\on}(t)    =   \sum\limits_{m=0}^\infty \gamma_m(A) (1 - t)^{m - d},
\end{equation}
where $d =2(n-2)$ is the Krull dimension of the algebra $\R[V]^{\T^2}/I_J^{\T^2}$,
and compute explicit formulas for $\gamma_0$ and $\gamma_2$.
By the proof of \cite[Theorem~1.3]{HerbigHerdenSeaton}, the algebra $\R[V]^{\T^2}/I_J^{\T^2}$ is \emph{graded Gorenstein}, which in
particular implies that $\gamma_1 = 0$ and $\gamma_2 = \gamma_3$; see
\cite[Definition~1.1 and Corollary~1.8]{HerbigHerdenSeaton} or \cite[Theorem~1.1]{HerbigHerdenSeaton2}.

Our approach is to compute the Laurent coefficients of $H_X(t)$ for a choice of $X$ such that
$c_{ij} + c_{ik} + c_{jk} \neq 0$ for each distinct $i,j,k$. Hence, we will need the following result to
extend our computations to the limit as $X\to A$.

\begin{lemma}
\label{lem:GamConverge}
Let $f_{\mathbf{x}}(t)$ be a family of meromorphic function depending continuously on finitely many parameters
$\mathbf{x} = (x_1,\ldots,x_m)$. Let $t_0\in\C$, and assume that there are open neighborhoods $O$ of
$t_0$ in $\C$ and $U$ of $\mathbf{a} = (a_1,\ldots,a_m)$ in $\C^m$ such that for all $\mathbf{x}\in U$, the only pole of $f_{\mathbf{x}}(t)$
in $O$ is at $t = t_0$. Then for each $d \in \Z$, the degree $d$ Laurent coefficient of $f_{\mathbf{x}}(t)$ at $t = t_0$
converges to the degree $d$ Laurent coefficient of $f_{\mathbf{a}}(t)$ at $t = t_0$ as $\mathbf{x}\to\mathbf{a}$.
\end{lemma}
\begin{proof}
Let $P$ be a simple closed positively-oriented curve in $O$ about $t_0$ and let $d\in\Z$. Then the degree $d$
Laurent coefficient of $f_{\mathbf{x}}(t)$ at $t_0$ is given by
\[
    \frac{1}{2\pi\sqrt{-1}} \int_P \frac{f_{\mathbf{x}}(t)\, dt}{(t - t_0)^{d+1}}.
\]
Let $D \subset U$ be the closure of a neighborhood of $\mathbf{a}$ in $\C^m$, and then as $P\times D$ is compact,
the continuous function $f_{\mathbf{x}}(t)$ is bounded on $P\times D$. Then by the Dominated Convergence Theorem,
we have
\[
    \lim\limits_{\mathbf{x}\to\mathbf{a}}
        \frac{1}{2\pi\sqrt{-1}} \int_P \frac{f_{\mathbf{x}}(t)\, dt}{(t - t_0)^{d+1}}
        =   \frac{1}{2\pi\sqrt{-1}} \int_P \frac{f_{\mathbf{a}}(t)\, dt}{(t - t_0)^{d+1}},
\]
completing the proof.
\end{proof}


\subsection{The first Laurent coefficient}
\label{subsec:Gam0}

Here, we consider the coefficient $\gamma_0(A)$ in the expansion \eqref{eq:GamDef} and prove the following.

\begin{theorem}
\label{thrm:Gam0}
Let $n > 2$ and let $A\in\Z^{2\times n}$ be a faithful weight matrix in standard form. The pole order of
$\Hilb_A^{\on}(t)$ at $t = 1$ is $2n - 4$, and the first nonzero Laurent coefficient $\gamma_0(A)$ of
$\Hilb_A^{\on}(t)$ is given by
\begin{equation}
\label{eq:Gam0}
    \gamma_0(A) =
    \lim\limits_{X\to A}
    \sum\limits_{\substack{i\neq j, \\ d_{ij} > 0}} \frac{ c_{ij}^{2n-5} }
        { \prod\limits_{k\neq i,j} (c_{ij} - c_{ik} - c_{jk})(c_{ij} + c_{ik} + c_{jk}) },
\end{equation}
where the $x_{ij}$ are real parameters approximating the $a_{ij}$, $X = (x_{ij})$, and
$c_{ij} = x_{1i} x_{2j} - x_{2i} x_{1j}$. In particular, for each $i,j,k$ such that $d_{ij} > 0$,
the singularities in Equation~\eqref{eq:Gam0} corresponding to $d_{ij} - d_{ik} - d_{jk} = 0$ and
$d_{ij} + d_{ik} + d_{jk} = 0$ are removable.

For the special case of a generic weight matrix $A$, we have the simplified formula
\begin{equation}
    \gamma_0(A) =
        \sum\limits_{\substack{i\neq j, \\ d_{ij} > 0}} \frac{ d_{ij}^{2n-5} }
        { \prod\limits_{k\neq i,j} (d_{ij} - d_{ik} - d_{jk})(d_{ij} + d_{ik} + d_{jk}) }.
\end{equation}
\end{theorem}

Throughout this section, we fix $x_{ij}$ and corresponding $c_{ij}$ such that each $c_{ij} - c_{ik} - c_{jk} \neq 0$
and each $c_{ij} + c_{ik} + c_{jk} \neq 0$.
For each fixed $i\neq j$ such that $d_{ij} > 0$ and $d_{ij}$th roots of unity $\xi$ and $\zeta$, the pole order of
the term $H_{X,i,j,\xi,\zeta}(t)$ given by Equation~\eqref{eq:HilbTerm}
is equal to $2(n - s - 2)$ where $s = s(i,j,\xi,\zeta)$ is the number of $k\neq i,j$ such that
$\xi^{d_{ik}} \zeta^{d_{jk}} \neq 1$. The maximum pole order is $2n - 4$, which occurs for instance when
$\xi = \zeta = 1$. A term has a pole of order $2n - 4$,
and hence contributes to $\gamma_0$, if and only if $\xi^{d_{ik}} \zeta^{d_{jk}} = 1$ for each $k\neq i,j$.

Now, fix $i\neq j$ with $d_{ij} > 0$. By Proposition~\ref{prop:NumSolsCongruence}, the number of pairs $(\xi, \zeta)$
of $d_{ij}$th roots of unity such that $\xi^{d_{ik}} \zeta^{d_{jk}} = 1$ for all $k\neq i,j$ is equal to $d_{ij}$.
For each such $(\xi,\zeta)$, we have
\[
    H_{X,i,j,\xi,\zeta}(t)  = \frac{1}{c_{ij} d_{ij}
            \prod\limits_{k\neq i,j}
                \big(1 - t^{(c_{ij} + c_{ik} + c_{jk})/c_{ij} } \big)
                \big(1 - t^{(c_{ij} - c_{ik} - c_{jk})/c_{ij} } \big)  },
\]
implying that
\begin{equation}
\label{eq:Gam0Term}
    \sum\limits_{\substack{\xi^{d_{ij}}=\zeta^{d_{ij}} = 1\\ \forall k\ne i,j: \xi^{d_{ik}} \zeta^{d_{jk}} = 1}}
    H_{X,i,j,\xi,\zeta}(t)
    =
    \frac{1}{c_{ij} \prod\limits_{k\neq i,j}
                \big(1 - t^{(c_{ij} + c_{ik} + c_{jk})/c_{ij} } \big)
                \big(1 - t^{(c_{ij} - c_{ik} - c_{jk})/c_{ij} } \big)  }.
\end{equation}
Recalling that we define $t^y$ using a fixed branch of $\log t$ such that $\log 1 = 0$, we have the Laurent expansion
\begin{equation}
\label{eq:TaylorNoRoot1}
    \frac{1}{1 - t^y}
    =
    \frac{1}{y(1 - t)} + \frac{y - 1}{2y} + \frac{y^2 - 1}{12y}(1 - t) + \cdots.
\end{equation}
Hence, the degree $4 - 2n$ coefficient of the Laurent series of \eqref{eq:Gam0Term} at $t = 1$ is given by
\[
    \frac{ c_{ij}^{2n-5} }
        { \prod\limits_{k\neq i,j} (c_{ij} - c_{ik} - c_{jk})(c_{ij} + c_{ik} + c_{jk}) }.
\]
Summing over $i\neq j$ such that $d_{ij}> 0$, yields the following.

\begin{lemma}
\label{lem:Gam0NoLimit}
Assume $c_{ij} + c_{ik} + c_{jk} \neq 0$ for each distinct $i,j,k$. Then the degree $4 - 2n$ coefficient of the
Laurent series of the function $H_X(t)$ defined in Equation~\eqref{eq:HilbSumTerms} is given by
\begin{equation}
\label{eq:Gam0NoLimit}
    \sum\limits_{\substack{i\neq j, \\ d_{ij} > 0}} \frac{ c_{ij}^{2n-5} }
        { \prod\limits_{k\neq i,j} (c_{ij} - c_{ik} - c_{jk})(c_{ij} + c_{ik} + c_{jk}) }.
\end{equation}
\end{lemma}

Observe in Equations~\eqref{eq:HilbTerm} and \eqref{eq:HilbSumTerms} that there are finitely many values of
$\xi^{d_{ik}}\zeta^{d_{jk}}$, and hence that there is an open neighborhood $O$ of $t = 1$ in $\C$ such that when
the $x_{ij}$ are sufficiently close to the $a_{ij}$, the only pole of $H_X(t)$ in $O$ occurs at $t = 1$. Hence,
Theorem~\ref{thrm:Gam0} follows from Lemmas~\ref{lem:GamConverge} and \ref{lem:Gam0NoLimit} and
Equation~\eqref{eq:HilbSumTermsZ}.


\subsection{Cancellations in the first Laurent coefficient}
\label{subsec:Cancellations}

In the case of the one-dimensional torus considered in \cite{HerbigSeaton} and \cite{CowieHerbigSeatonHerden},
the first Laurent coefficient $\gamma_0$ is given by an expression similar to Equation~\eqref{eq:Gam0}.
In that case, the removability of the singularities was understood by interpreting this expression as
the quotient of a determinant, which was therefore divisible by the Vandermonde determinant in
the weights. The result is a description of the numerator after the cancellations as a Schur polynomial,
and hence a closed form expression for $\gamma_0$. This in particular leads to a quick proof that $\gamma_0$
is always positive.

In the case at hand, Theorem~\ref{thrm:Gam0} guarantees that the singularities in the expression
for $\gamma_0(A)$ in Equation~\eqref{eq:Gam0} are removable, just as in the one-dimensional case.
However, we have not obtained a similar combinatorial description of the expression for $\gamma_0(A)$
after the cancellations. In particular, we conjecture that $\gamma_0(A)$ is always positive for a faithful
weight matrix in standard form, and such a description would be useful to prove this claim. It could as
well lead to more efficient computation of the $\gamma_m(A)$ for specific degenerate $A$.

In this section, we describe an approach to understanding these cancellations
through brute force computations and end with a discussion of small values of $n$. See
Section~\ref{subsec:GamAlg} for a method to compute $\gamma_0(A)$ for specific examples of degenerate
weight matrices $A$ without having to perform the cancellations in general.

First, let us be more explicit about the cancellations in Equation~\eqref{eq:Gam0NoLimit}. We combine the sum in
Equation~\eqref{eq:Gam0NoLimit} into a single rational function of the form
\begin{equation}
\label{eq:Gam0Combined}
    \frac{\sum\limits_{\substack{i\neq j, \\ d_{ij} > 0}}  c_{ij}^{2n - 5}
        \prod\limits_{\substack{p,q,r \\ d_{pq}, d_{qr} > 0 \\ (p,q)\neq(i,j) \\ (q,r)\neq(i,j)}}
            (c_{pq} + c_{pr} + c_{qr})
        \prod\limits_{\substack{p,q,r \\ d_{pq}, d_{qr} > 0 \\ (p,r)\neq(i,j) \\ (q,r)\neq(i,j)}}
            (c_{qr} - c_{pq} - c_{pr})
        \prod\limits_{\substack{p,q,r \\ d_{pq}, d_{qr} > 0 \\ (p,q)\neq(i,j) \\ (p,r)\neq(i,j)}}
            (c_{pq} - c_{pr} - c_{qr})
            }
            {\prod\limits_{\substack{p,q,r \\ d_{pq}, d_{qr} > 0}}
                (c_{pq} + c_{pr} + c_{qr})(c_{qr} - c_{pq} - c_{pr})(c_{pq} - c_{pr} - c_{qr})}.
\end{equation}
Note that we continue to express the limits of the products and sums in terms of the $d_{ij}$ to emphasize
that the signs of the $c_{ij}$ and $d_{ij}$ coincide. Note further that if $c_{pq} > 0$ and $c_{qr} \geq 0$,
the hypothesis that $A$ is in standard form so that each $a_{1i} > 0$ (and hence each $x_{1i} > 0$) implies that
$c_{pr} \geq 0$ as well. That is, the factors of the form $(c_{pq} + c_{pr} + c_{qr})$ are always positive when
the $x_{ij}$ approximate a weight matrix in standard form, and only the other singularities are relevant.
Hence, the cancellations amount to the numerator of Equation~\eqref{eq:Gam0Combined} being
divisible by the factors of the form $(c_{qr} - c_{pq} - c_{pr})$ and $(c_{pq} - c_{pr} - c_{qr})$,
and the desired combinatorial description is an expression for the polynomial
\begin{equation}
\label{eq:Gam0CombinedNumerator}
    \frac{\sum\limits_{\substack{i\neq j, \\ d_{ij} > 0}}  c_{ij}^{2n - 5}
        \prod\limits_{\substack{p,q,r \\ d_{pq}, d_{qr} > 0 \\ (p,q)\neq(i,j) \\ (q,r)\neq(i,j)}}
            (c_{pq} + c_{pr} + c_{qr})
        \prod\limits_{\substack{p,q,r \\ d_{pq}, d_{qr} > 0 \\ (p,r)\neq(i,j) \\ (q,r)\neq(i,j)}}
            (c_{qr} - c_{pq} - c_{pr})
        \prod\limits_{\substack{p,q,r \\ d_{pq}, d_{qr} > 0 \\ (p,q)\neq(i,j) \\ (p,r)\neq(i,j)}}
            (c_{pq} - c_{pr} - c_{qr})
            }
            {\prod\limits_{\substack{p,q,r \\ d_{pq}, d_{qr} > 0}}
                (c_{qr} - c_{pq} - c_{pr})(c_{pq} - c_{pr} - c_{qr})}.
\end{equation}
The goal of this section is to give an alternate and more explicit demonstration that this is indeed a polynomial.

It is important to recall that the $c_{ij}$ are not independent variables; due to their
dependence on the $x_{ij}$ they satisfy the Pl\"{u}cker relations, see Equation~\eqref{eq:Pluecker}.
In general, the singularities in question are only removable if Equation~\eqref{eq:Gam0Combined} is interpreted as
a function in the $x_{ij}$ rather than treating the $c_{ij}$ as independent variables.

\begin{lemma}
\label{lem:Gam0RemovSing1}
As a polynomial in the $x_{ij}$, the numerator of Equation~\eqref{eq:Gam0Combined} is divisible by
the product of $(c_{pq} - c_{pr} - c_{qr})$ such that $d_{pq}, d_{qr} > 0$.
\end{lemma}
\begin{proof}
Let $F$ denote the numerator of Equation~\eqref{eq:Gam0Combined}, and let $F_{ij}$ denote the summand of $F$
corresponding to $i\neq j$ such that $d_{ij} > 0$. Note that for fixed $p,q,r$, the polynomial
$(c_{pq} - c_{pr} - c_{qr})$ (as a quadratic polynomial in the $x_{ij}$ has no linear factors is hence is reducible.
It is therefore sufficient to show that $F$ is divisible by each such factor individually.

Pick $I,J,K$ such that $d_{IJ}, d_{JK} > 0$. We will demonstrate that $F$
is contained in the ideal generated by the Pl\"{u}cker relations and $(c_{IJ} - c_{IK} - c_{JK})$. First note that
each summand of $F$ contains $(c_{IJ} - c_{IK} - c_{JK})$ explicitly as a factor except $F_{IJ}$ and $F_{IK}$
so that we may restrict our attention to $F_{IJ} + F_{IK}$. Both of these summands contain the factors
\[
    \prod\limits_{\substack{d_{pq}, d_{qr} > 0 \\ (p,q),(q,r)\neq \\ (I,J), (I,K)}}
        (c_{pq} + c_{pr} + c_{qr})
    \prod\limits_{\substack{d_{pq}, d_{qr} > 0 \\ (p,r),(q,r)\neq \\ (I,J), (I,K)}}
        (c_{qr} - c_{pq} - c_{pr})
    \prod\limits_{\substack{d_{pq}, d_{qr} > 0 \\ (p,q),(p,r)\neq \\ (I,J), (I,K)}}
        (c_{pq} - c_{pr} - c_{qr}),
\]
so that we may express $F_{IJ} + F_{IK}$ as a product of the above polynomial and the remaining factors,
where the latter can be expressed as
\begin{align*}
    &c_{IJ}^{2n - 5}
        (c_{JK} - c_{IJ} - c_{IK})
        \prod\limits_{\substack{r, \\ d_{Kr} > 0}}
            (c_{IK} + c_{Ir} + c_{Kr})
            (c_{IK} - c_{Ir} - c_{Kr})
    \\&\quad\quad
        \prod\limits_{\substack{p, \\  d_{pI} > 0}}
            (c_{pI} + c_{pK} + c_{IK})
            (c_{IK} - c_{pI} - c_{pK})
        \prod\limits_{\substack{q \neq J, \\ d_{Iq}, d_{qK} > 0}}
            (c_{qK} - c_{Iq} - c_{IK})
            (c_{Iq} - c_{IK} - c_{qK})
    \\&\quad
    + c_{IK}^{2n - 5}
        (c_{IJ} + c_{IK} + c_{JK})
        \prod\limits_{\substack{r \neq K \\ d_{Jr} > 0}}
            (c_{IJ} + c_{Ir} + c_{Jr})
            (c_{IJ} - c_{Ir} - c_{Jr})
    \\&\quad\quad
        \prod\limits_{\substack{p, \\ d_{pI} > 0}}
            (c_{pI} + c_{pJ} + c_{IJ})
            (c_{IJ} - c_{pI} - c_{pJ})
        \prod\limits_{\substack{q, \\ d_{Iq}, d_{qJ} > 0}}
            (c_{qJ} - c_{Iq} - c_{IJ})
            (c_{Iq} - c_{IJ} - c_{qJ}).
\end{align*}
In the first summand, it is helpful to think of the indices $p$, $q$, and $r$ as ranging over those columns of $A$
that, as vectors in $\R^2$, lie below $(a_{1I},a_{2I})$ (for $p$), between the $(a_{1I},a_{2I})$ and
$(a_{1K},a_{2K})$ (for $q$), and above $(a_{1J},a_{2J})$ (for $r$), and similarly for the second summand.
Applying the substitution $c_{IJ} = c_{IK} + c_{JK}$ in the factor $(c_{JK} - c_{IJ} - c_{IK})$ in the first
summand yields $- 2c_{IK}$, and applying $c_{IK} = c_{IJ} - c_{JK}$ in $(c_{IJ} + c_{IK} + c_{JK})$ in the second summand
yields $2c_{IJ}$. Noting that each index value not equal to $I,K,J$ appears exactly once as a $p$, $q$, or $r$ in each of
the above summands so that the total number of three-term factors in each summand is $2n - 6$, we express this as
\begin{align*}
    -2 c_{IJ} c_{IK}
        &\prod\limits_{r, \: d_{Kr} > 0}
            c_{IJ}(c_{IK} + c_{Ir} + c_{Kr})
            c_{IJ}(c_{IK} - c_{Ir} - c_{Kr})
    \\&\prod\limits_{p, \: d_{pI} > 0}
            c_{IJ}(c_{pI} + c_{pK} + c_{IK})
            c_{IJ}(c_{IK} - c_{pI} - c_{pK})
    \\&\prod\limits_{\substack{q \neq J, \\ d_{Iq}, d_{qK} > 0}}
            c_{IJ}(c_{qK} - c_{Iq} - c_{IK})
            c_{IJ}(c_{Iq} - c_{IK} - c_{qK})
    \\\quad
    + 2 c_{IJ} c_{IK}
        &\prod\limits_{\substack{r \neq K, \\ d_{Jr} > 0}}
            c_{IK}(c_{IJ} + c_{Ir} + c_{Jr})
            c_{IK}(c_{IJ} - c_{Ir} - c_{Jr})
    \\&\prod\limits_{p, \: d_{pI} > 0}
            c_{IK}(c_{pI} + c_{pJ} + c_{IJ})
            c_{IK}(c_{IJ} - c_{pI} - c_{pJ})
    \\&\prod\limits_{q, \: d_{Iq}, d_{qJ} > 0}
            c_{IK}(c_{qJ} - c_{Iq} - c_{IJ})
            c_{IK}(c_{Iq} - c_{IJ} - c_{qJ}).
\end{align*}
We will now rewrite the first summand to see that it is equal to the negative of the second summand.
Distributing a $c_{IJ}$ into each three-term factor, we apply the Pl\"{u}cker relations
\eqref{eq:Pluecker} $c_{IJ}c_{Kr} - c_{IK}c_{Jr} + c_{Ir} c_{JK} = 0$,
$c_{pI}c_{JK} - c_{pJ}c_{IK} + c_{pK} c_{IJ} = 0$, and $c_{Iq}c_{JK} - c_{IJ}c_{qK} + c_{IK} c_{qJ} = 0$,
so the first summand becomes
\begin{align*}
    -2 c_{IJ} c_{IK}
        &\prod\limits_{r, \: d_{Kr} > 0}
            (c_{IJ}c_{IK} + c_{IJ}c_{Ir} + c_{IK}c_{Jr} - c_{Ir} c_{JK})
            (c_{IJ}c_{IK} - c_{IJ}c_{Ir} - c_{IK}c_{Jr} + c_{Ir} c_{JK})
    \\&\prod\limits_{p, \: d_{pI} > 0}
            (c_{IJ}c_{pI} - c_{pI}c_{JK} + c_{pJ}c_{IK} + c_{IJ}c_{IK})
            (c_{IJ}c_{IK} - c_{IJ}c_{pI} + c_{pI}c_{JK} - c_{pJ}c_{IK})
    \\&\prod\limits_{\substack{q\neq J \\ d_{Iq}, d_{qK} > 0}}
            (c_{Iq}c_{JK} + c_{IK} c_{qJ} - c_{IJ}c_{Iq} - c_{IJ}c_{IK})
            (c_{IJ}c_{Iq} - c_{IJ}c_{IK} - c_{Iq}c_{JK} - c_{IK} c_{qJ}).
\end{align*}
Applying the relation $c_{IJ} - c_{IK} - c_{JK} = 0$, we rewrite this as
\begin{align*}
    -2 c_{IJ} c_{IK}
        &\prod\limits_{r, \: d_{Kr} > 0}
            (c_{IJ}c_{IK} + c_{IK}c_{Ir} + c_{IK}c_{Jr})
            (c_{IJ}c_{IK} - c_{IK}c_{Ir} - c_{IK}c_{Jr})
    \\&\prod\limits_{p, \: d_{pI} > 0}
            (c_{IK}c_{pI} + c_{pJ}c_{IK} + c_{IJ}c_{IK})
            (c_{IJ}c_{IK} - c_{IK}c_{pI} - c_{pJ}c_{IK})
    \\&\prod\limits_{\substack{q\neq J, \\ d_{Iq}, d_{qK} > 0}}
            (c_{IK} c_{qJ} - c_{IK}c_{Iq} - c_{IJ}c_{IK})
            (c_{IK}c_{Iq} - c_{IJ}c_{IK} - c_{IK} c_{qJ}).
\end{align*}
Recalling the assumption that $c_{IJ}, c_{JK} > 0$ and reorganizing factors, this is equal to the negative of the
second summand, completing the proof.
\end{proof}

An almost identical argument yields the following.

\begin{lemma}
\label{lem:Gam0RemovSing2}
As a polynomial in the $x_{ij}$, the numerator of Equation~\eqref{eq:Gam0Combined} is divisible by
the product of $(c_{qr} - c_{pq} - c_{pr})$ such that $d_{pq}, d_{qr} > 0$.
\end{lemma}

Combining Lemmas \ref{lem:Gam0RemovSing1} and \ref{lem:Gam0RemovSing2}, it follows that
Equation~\eqref{eq:Gam0Combined}, as a rational function in the $x_{ij}$, can be expressed in the form
\begin{equation}
\label{eq:Gam0SAfterCancel}
    \frac{S}{\prod\limits_{\substack{p,q,r \\ d_{pq}, d_{qr} > 0}}(c_{pq} + c_{pr} + c_{qr})},
\end{equation}
where $S$ is a polynomial in the $x_{ij}$ that is equal to the expression in Equation~\eqref{eq:Gam0CombinedNumerator}
on its domain.

When $n=3$, the cancellations can be dealt with by hand; in this case, they occur even if the $c_{ij}$ for
$i < j$ are treated as independent variables (i.e., without applying the Pl\"{u}cker relations), and the resulting
numerator $S = 1$ is constant. When $n = 4$,
$S$ has $14$ terms in the $c_{ij}$; when $n = 5$, the cancellations involved a Gr\"{o}bner basis computation that took
five days on a PC and yielded an $S$ with $1961$ terms in the $c_{ij}$. Of course, when $n > 3$, the number of terms is
not unique due to the Pl\"{u}cker relations.


\subsection{The next three Laurent coefficients}
\label{subsec:Gam123}

We now turn to the computation of the next Laurent coefficients and prove the following.

\begin{theorem}
\label{thrm:Gam2}
Let $n > 2$ and let $A\in\Z^{2\times n}$ be a faithful weight matrix in standard form. Then $\gamma_1(A) = 0$ and
\begin{equation}
\label{eq:Gam2}
    \gamma_2(A) = \gamma_3(A) =
    \lim\limits_{X\to A}
    \sum\limits_{\substack{i\neq j, \\ d_{ij} > 0}}
        \frac{- c_{ij}^{2n - 7} \sum\limits_{p\neq i,j} (c_{ip} + c_{jp})^2}
        {12\prod\limits_{k\neq i,j} (c_{ij} - c_{ik} - c_{jk})(c_{ij} + c_{ik} + c_{jk})}
    + \sum\limits_{p=1}^n \frac{g_p - 1}{12} \gamma_0(A_p),
\end{equation}
where the $x_{ij}$ are real parameters approximating the $a_{ij}$, $X = (x_{ij})$,
$c_{ij} = x_{1i} x_{2j} - x_{2i} x_{1j}$, $A_p$ is the weight matrix formed by removing column $p$
from $A$, and $g_p$ is the $\gcd$ of the $2\times 2$ minors of $A_p$.
In particular, for each $i,j,k$ such that $d_{ij} > 0$, the singularities in Equation~\eqref{eq:Gam2}
corresponding to $d_{ij} - d_{ik} - d_{jk} = 0$ and $d_{ij} + d_{ik} + d_{jk} = 0$ are removable.

For the special case of a generic weight matrix $A$, we have the simplified formula
\begin{equation}
    \gamma_2(A) = \gamma_3(A) =
    \sum\limits_{\substack{i\neq j, \\ d_{ij} > 0}}
        \frac{- d_{ij}^{2n - 7} \sum\limits_{p\neq i,j} (d_{ip} + d_{jp})^2}
        {12\prod\limits_{k\neq i,j} (d_{ij} - d_{ik} - d_{jk})(d_{ij} + d_{ik} + d_{jk})}
    + \sum\limits_{p=1}^n \frac{g_p - 1}{12} \gamma_0(A_p).
\end{equation}
\end{theorem}
\begin{proof}
We first compute $\gamma_1(A)$. Note that the fact that $\gamma_1(A) = 0$ follows from the results
of \cite{HerbigHerdenSeaton} as noted after Equation~\eqref{eq:GamDef} above. We verify this explicitly here on
the way towards the computation of $\gamma_2(A)$.

Based on the observations after the statement of Theorem~\ref{thrm:Gam0} and continuing to use the same
notation, we need to consider terms $H_{X,i,j,\xi,\zeta}(t)$ of pole order $2(n - s - 2)$ where $s = 0$ or $1$;
recall that $s = s(i,j,\xi,\zeta)$ denotes the number of $k$ such that $\xi^{d_{ik}} \zeta^{d_{jk}} \neq 1$. In
particular, there are no terms with pole order $2n - 3$, so only terms where $s = 0$ contribute to $\gamma_1(A)$.

Using Equation~\eqref{eq:TaylorNoRoot1} and the Cauchy product formula, the degree
$5 - 2n$ coefficient of the Laurent expansion of the term $H_{X,i,j,\xi,\zeta}(t)$ is computed by choosing
a $p\neq i,j$, multiplying the degree $0$ coefficient of the expansion of
$1/(1 - t^{(c_{ij} + c_{ip} + c_{jp})/c_{ij} })$ or
$1/(1 - t^{(c_{ij} - c_{ip} - c_{jp})/c_{ij} })$ with the degree $-1$ coefficient of the other factors,
and summing over all choices of $p$ and the factor from the pair. That is, the degree $5 - 2n$ coefficient
of a term $H_{X,i,j,\xi,\zeta}(t)$ such that $s = 0$ is given by
\begin{align*}
    \sum\limits_{p\neq i,j}
        \frac{ c_{ij}^{2n - 6}}
        {c_{ij}^2\prod\limits_{k\neq i,j,p} (c_{ij} - c_{ik} - c_{jk})(c_{ij} + c_{ik} + c_{jk})}
        \Bigg( \left(\frac{c_{ij}}{c_{ij} + c_{ip} + c_{jp}}\right)
        \frac{(c_{ij} - c_{ip} - c_{jp})/c_{ij} - 1}{2(c_{ij} - c_{ip} - c_{jp})/c_{ij}}
    \\+ \left(\frac{c_{ij}}{c_{ij} - c_{ip} - c_{jp}}\right)
        \frac{(c_{ij} + c_{ip} + c_{jp})/c_{ij} - 1}{2(c_{ij} + c_{ip} + c_{jp})/c_{ij}}
        \Bigg)
    =   0,
\end{align*}
confirming that $\gamma_1(A) = 0$ (which also follows from the results of
\cite{HerbigHerdenSeaton,HerbigHerdenSeaton2} as described above).

For the degree $6 - 2n$ coefficient, we first consider the contribution of terms $H_{X,i,j,\xi,\zeta}(t)$ with
$s = 0$. The contribution is computed similarly to above, except that
we consider the products of the degree $1$ coefficient of a factor corresponding to $p\neq i,j$ with the degree $-1$
coefficients of the other factors, and the degree $0$ coefficient of two factors corresponding to $p,q\neq i,j$
with the degree $-1$ coefficients of the other factors.

In the first case, we have
\begin{align*}
    \sum\limits_{p\neq i,j}
        \frac{ c_{ij}^{2n - 6}}
        {c_{ij}^2\prod\limits_{k\neq i,j,p} (c_{ij} - c_{ik} - c_{jk})(c_{ij} + c_{ik} + c_{jk})}
        \Bigg( \left(\frac{c_{ij}}{c_{ij} + c_{ip} + c_{jp}}\right)
        \frac{(c_{ij} - c_{ip} - c_{jp})^2/c_{ij}^2 - 1}{12(c_{ij} - c_{ip} - c_{jp})/c_{ij}}
    \\+ \left(\frac{c_{ij}}{c_{ij} - c_{ip} - c_{jp}}\right)
        \frac{(c_{ij} + c_{ip} + c_{jp})^2/c_{ij}^2 - 1}{12(c_{ij} + c_{ip} + c_{jp})/c_{ij}}
        \Bigg)
    \\= \frac{ c_{ij}^{2n - 8} \sum\limits_{p\neq i,j} (c_{ip} + c_{jp})^2}
        {6 \prod\limits_{k\neq i,j} (c_{ij} - c_{ik} - c_{jk})(c_{ij} + c_{ik} + c_{jk})}.
\end{align*}
Summing over all terms $H_{X,i,j,\xi,\zeta}(t)$ and recalling that for each $i\neq j$ with $d_{ij} > 0$, there
are by Proposition~\ref{prop:NumSolsCongruence} $d_{ij}$ pairs $(\xi,\zeta)$ such that $s = 0$, we have
\begin{equation}
\label{eq:Gam2r0OneExcep}
    \sum\limits_{\substack{i\neq j,\\d_{ij} > 0}}
        \frac{ c_{ij}^{2n - 7} \sum\limits_{p\neq i,j} (c_{ip} + c_{jp})^2}
        {6\prod\limits_{k\neq i,j} (c_{ij} - c_{ik} - c_{jk})(c_{ij} + c_{ik} + c_{jk})}.
\end{equation}

In the second case, we first consider the situation where both factors
$1/(1 - t^{(c_{ij} + c_{ip} + c_{jp})/c_{ij} })$ and $1/(1 - t^{(c_{ij} - c_{ip} - c_{jp})/c_{ij} })$
contribute a degree $0$ coefficient for some $p\ne i,j$ while the remaining factors corresponding to
$k\neq i,j,p$ contribute their degree $-1$ coefficient. Summing over all relevant terms,
a calculation very similar to those above yields the contribution
\begin{equation}
\label{eq:Gam2r0TwoExcep}
    \sum\limits_{\substack{i\neq j,\\d_{ij} > 0}}
        \frac{- c_{ij}^{2n - 7} \sum\limits_{p\neq i,j} (c_{ip} + c_{jp})^2}
        {4\prod\limits_{k\neq i,j} (c_{ij} - c_{ik} - c_{jk})(c_{ij} + c_{ik} + c_{jk})}.
\end{equation}
In addition, we need to consider the situation where distinct $p,q\neq i,j$ with $p<q$
contribute each the degree $0$ coefficient of one of their corresponding factors
$1/(1 - t^{(c_{ij} + c_{ir} + c_{jr})/c_{ij} })$ and $1/(1 - t^{(c_{ij} - c_{ir} - c_{jr})/c_{ij} })$
where $r = p$ or $q$. An easy calculation shows that in this case terms cancel, and the total contribution is zero.

We now consider the contribution of $H_{X,i,j, \xi,\zeta}(t)$ where $s = 1$, i.e., $\xi^{d_{ik}} \zeta^{d_{jk}} = 1$
for all $k\neq i,j$ except one, say $k = p$. Such a term corresponds to a choice of $i,j$ and a solution $(\xi,\zeta)$ to
$\xi^{d_{ik}} \zeta^{d_{jk}} = 1$, $k\neq i,j$ for the weight matrix $A_p\in\Z^{2\times(n-1)}$ formed by
removing the $p$th column. If $A_p$ is faithful, then the number of pairs $(\xi,\zeta)$ such
that $\xi^{d_{ik}} \zeta^{d_{jk}} = 1$ for all $k\neq i,j,p$ is $d_{ij}$ by Proposition~\ref{prop:NumSolsCongruence},
hence each such pair satisfies $\xi^{d_{ip}} \zeta^{d_{jp}} = 1$ by counting. It follows that there are no
$H_{X,i,j,\xi,\zeta}(t)$ with $s = 1$ corresponding to $A_p$. Similarly, if $A_p$ has rank $1$, then there are no
$i,j\neq p$ such that $d_{ij} > 0$ and hence no corresponding terms $H_{X,i,j,\xi,\zeta}(t)$ in $H_X(t)$.

If $A_p$ is not faithful and the $2\times 2$ minors of $A_p$ have $\gcd$ $g_p > 1$, then there are
$g_p d_{ij}$ pairs $(\xi,\zeta)$ to consider such that $\xi^{d_{ik}} \zeta^{d_{jk}} = 1$ for all $k\neq i,j,p$,
again by Proposition~\ref{prop:NumSolsCongruence}.
Identifying the set of pairs of $d_{ij}$th roots of unity with $(\Z/d_{ij}\Z)^2$,
the set of $(\xi,\zeta)$ such that $\xi^{d_{ik}} \zeta^{d_{jk}} = 1$ for $k\neq i,j,p$ forms a subgroup
of order $g_pd_{ij}$. As $\xi^{d_{ip}} \zeta^{d_{jp}} = 1$ for $d_{ij}$ of these pairs,
$(\xi,\zeta)\mapsto \xi^{d_{ip}} \zeta^{d_{jp}}$ is a homomorphism to $\Z/d_{ij}\Z$
with kernel of order $d_{ij}$. Therefore, the image of this homomorphism corresponds to a cyclic subgroup
of $\Z/d_{ij}\Z$ of size $g_p$, which means that the homomorphism $(\xi,\zeta)\mapsto \xi^{d_{ip}} \zeta^{d_{jp}}$
maps onto the group of $g_p$th roots of unity, and
for each $g_p$th root of unity $\eta$, there are $d_{ij}$ pairs $(\xi,\zeta)$ such that
$\xi^{d_{ip}} \zeta^{d_{jp}} = \eta$.

Fixing $i,j,p$ and $(\xi,\zeta)$ such that $\xi^{d_{ip}} \zeta^{d_{jp}} =\eta \neq 1$,
$H_{X,i,j,\xi,\zeta}(t)$ is of the form
\begin{align*}
    H_{X,i,j,\xi,\zeta}(t)  = \Big( c_{ij}^2
            \big(1 - \eta t^{(c_{ij} + c_{ip} + c_{jp})/c_{ij} } \big)
                \big(1 - \eta^{-1} t^{(c_{ij} - c_{ip} - c_{jp})/c_{ij} } \big)
            \\ \prod\limits_{k\neq i,j,p}
                \big(1 - t^{(c_{ij} + c_{ik} + c_{jk})/c_{ij} } \big)
                \big(1 - t^{(c_{ij} - c_{ik} - c_{jk})/c_{ij} } \big)  \Big)^{-1}
\end{align*}
and has a pole order of $2n - 6$. Using the expansion \eqref{eq:TaylorNoRoot1} as well as
\[
    \frac{1}{1 - \nu t^y}
    =
    \frac{1}{1 - \nu} - \frac{\nu y}{(\nu - 1)^2}(1 - t) + \cdots,
\]
the coefficient of degree $2n - 6$ of $H_{X,i,j,\xi,\zeta}(t)$ is given by
\begin{align*}
    \frac{c_{ij}^{2n-6}}{c_{ij}^2(1 - \eta)(1 - \eta^{-1})
        \prod\limits_{k\neq i,j,p} (c_{ij} + c_{ik} + c_{jk})(c_{ij} - c_{ik} - c_{jk}) }
    \\=
    \frac{-c_{ij}^{2n-8} }
        {\prod\limits_{k\neq i,j,p} (c_{ij} + c_{ik} + c_{jk})(c_{ij} - c_{ik} - c_{jk}) }
    \left(\frac{ \eta }{ (1 - \eta)^2 } \right).
\end{align*}
Summing over all $(\xi,\zeta)$ corresponding to the fixed $i,j,p$ such that $\xi^{d_{ip}} \zeta^{d_{jp}} \neq 1$,
we have
\begin{align*}
    \sum\limits_{\substack{\xi^{d_{ij}} = \zeta^{d_{ij}} = 1 \\ \xi^{d_{ip}} \zeta^{d_{jp}} = \eta \neq 1}}
    &\frac{-c_{ij}^{2n-8} }
        {\prod\limits_{k\neq i,j,p} (c_{ij} + c_{ik} + c_{jk})(c_{ij} - c_{ik} - c_{jk}) }
    \left(\frac{ \eta }{ (1 - \eta)^2 } \right)
    \\&=
    \frac{-c_{ij}^{2n-7} }
        {\prod\limits_{k\neq i,j,p} (c_{ij} + c_{ik} + c_{jk})(c_{ij} - c_{ik} - c_{jk}) }
    \sum\limits_{\eta^{g_p} = 1, \eta\neq 1}
        \left(\frac{ \eta }{ (1 - \eta)^2 } \right)
    \\&=
    \frac{c_{ij}^{2n-7} (g_p^2 - 1)}
        {12\prod\limits_{k\neq i,j,p} (c_{ij} + c_{ik} + c_{jk})(c_{ij} - c_{ik} - c_{jk}) },
\end{align*}
where the sum over $\eta$ is computed using \cite[Equation (3.11)]{Gessel}.
Summing over each $p$ and $i\neq j$ with $d_{ij} > 0$, we obtain
\[
    \sum\limits_{p=1}^n \frac{g_p^2 - 1}{12} \sum\limits_{\substack{i\neq j, \\ d_{ij} > 0}}
    \frac{c_{ij}^{2n-7}} {\prod\limits_{k\neq i,j,p} (c_{ij} + c_{ik} + c_{jk})(c_{ij} - c_{ik} - c_{jk}) }
    =   \sum\limits_{p=1}^n \frac{g_p^2 - 1}{12}\gamma_0(A_p).
\]
Combining this with Equations~\eqref{eq:Gam2r0OneExcep} and \eqref{eq:Gam2r0TwoExcep}, and applying
Lemma~\ref{lem:GamConverge} identically as in the proof of Theorem~\ref{thrm:Gam0},
completes computation of $\gamma_2(A)$. That $\gamma_2(A) = \gamma_3(A)$ follows from
\cite[Theorem~1.3 and Corollary~1.8]{HerbigHerdenSeaton}.
\end{proof}

\begin{remark}
\label{rem:Gam2Cancel}
As discussed in Section~\ref{subsec:Cancellations} for $\gamma_0(A)$, a combinatorial description of the
expression for $\gamma_2(A)$ in Theorem~\ref{thrm:Gam2} after the cancellations is desirable. The
cancellations in the second sum involving $\gamma_0(A_p)$ are as described in Lemmas~\ref{lem:Gam0RemovSing1}
and \ref{lem:Gam0RemovSing2}, and we have verified that the cancellations for $n = 3, 4$ occur
analogously in the first sum.
\end{remark}

\begin{remark}
\label{rem:OffShellGam0}
Using Equation~\eqref{eq:OnOffHilb}, an immediate consequence of Theorems~\ref{thrm:Gam0} and \ref{thrm:Gam2}
is an explicit formula for the first four Laurent coefficients $\gamma_m^{\off}(A)$ of the off-shell invariants;
see Remark~\ref{rem:OffShellHilb}. Specifically, the pole order of $\Hilb_A^{\off}(t)$ at $t = 1$ is
$2n - 2$, and we have the following
\begin{align*}
    \gamma_0^{\off}(A)  =   \gamma_1^{\off}(A)
                        &=  \frac{\gamma_0(A)}{4},
    \\
    \gamma_2^{\off}(A)  &=  \frac{3\gamma_0(A) + 4\gamma_2(A)}{16},
    \\
    \gamma_3^{\off}(A)  &=  \frac{\gamma_0(A) + 4\gamma_2(A)}{8}.
\end{align*}
\end{remark}


\subsection{Computing the Laurent coefficients}
\label{subsec:GamAlg}

In the case of a generic weight matrix $A$, Theorems~\ref{thrm:Gam0} and \ref{thrm:Gam2} can be used to compute
$\gamma_0(A)$ and $\gamma_2(A)$ with little difficulty. However, if $A$ has degeneracies, then as was noted in
Section~\ref{subsec:Cancellations}, an expression for $\gamma_0(A)$ with the singularities removed can be very expensive
to compute, even for representations as small as $n = 5$. Here, we briefly describe a method that has been
successful to more efficiently  compute $\gamma_0(A)$ for degenerate $A$ with values of $n$ as large as $10$.

Given a weight matrix $A\in\Z^{2\times n}$, the algorithm first tests that $A$ is faithful and in standard form,
and terminates if either of these hypotheses does not hold. Let
\[
    A(u_1,\ldots,u_n)
    =   \begin{pmatrix}
                a_{11}u_1  &   a_{12}u_2  &   \cdots  &   a_{1n}u_n  \\
                a_{21}u_1  &   a_{22}u_2  &   \cdots  &   a_{2n}u_n
            \end{pmatrix}
\]
so that $A(1,\ldots,1) = A$, and let $D_{pq}(u_1,\ldots,u_n) =d_{pq}u_pu_q$ denote the minor of $A(u_1,\ldots,u_n)$
corresponding to columns $p$ and $q$. For each $1\leq i,j\leq n$ with $d_{ij} > 0$, the denominator
$\prod_{k\neq i,j} (d_{ij} - d_{ik} - d_{jk})(d_{ij} + d_{ik} + d_{jk})$ of the corresponding term in
Equation~\eqref{eq:Gam0} with the substitution $c_{pq} = d_{pq}$ for each $p,q$ is evaluated. If the denominator
is nonzero, then the term is computed directly from the matrix with the above substitutions. If the denominator
vanishes, then the term is computed by substituting $c_{pq} = D_{pq}(u_1,\ldots,u_n)$ for each $p,q$. The sum of
the resulting terms is combined into a single rational fraction of the form in Equation~\eqref{eq:Gam0Combined}
in the indeterminates $u_1,\ldots,u_n$ with many of the singularities in that expression already removed.
The remaining singularities can be removed by factoring and cancelling or by polynomial division of the numerator
by the principal ideal generated by the product of factors of the denominator that vanish when each $u_i = 1$.

This method has been implemented on \emph{Mathematica} \cite{Mathematica} and is available
from the authors upon request. Unlike the algorithm to compute the complete Hilbert series described in
Section~\ref{subsec:Algorithm}, it has the benefit of not being as sensitive to the size of the
entries of $A$. For weight matrices with no degeneracies, it is simply arithmetic and hence fast,
and the computational expense grows with the number of degeneracies and only slowly with the $n$
and the size of the weights. It has successfully computed $\gamma_0(A)$ for weight matrices as
large as $2\times 10$ with multiple degeneracies in a matter of minutes.


\bibliographystyle{amsplain}
\bibliography{HHS-T2}

\end{document}